\documentclass[leqno, twosided]{amsart}
\usepackage{latexsym,amssymb,amsmath,graphics,enumerate,amsfonts}

\usepackage[dvips]{graphicx}
\usepackage[english]{babel}
\usepackage[T1]{fontenc}

\newcommand{\8}{\infty}
\renewcommand{\d}{\delta}
\renewcommand{\a}{\alpha}
\renewcommand{\b}{\beta}

\newcommand{\eps}{\varepsilon}

\newcommand{\ov}{\overline}

\newcommand{\is}[2]{\langle #1,#2\rangle}

\newtheorem{theorem}[equation]{Theorem}

\newtheorem{corollary}[equation]{Corollary}
\newtheorem{lem}[equation]{Lemma}
\newtheorem{lemma}[equation]{Lemma}

\newtheorem{prop}[equation]{Proposition}
\newtheorem{proposition}[equation]{Proposition}
\theoremstyle{definition}

\newtheorem{rem}{Remark}[section]
\newtheorem{remark}{Remark}[section]

\numberwithin{equation}{section}

\newcommand{\Cov}{\mathbb{K}}

\usepackage{color}

\allowdisplaybreaks

\newcommand{\E}{\mathbb{E}}
\newcommand{\Prob}{\mathbb{P}}

\renewcommand{\P}[2][]{\ensuremath{\mathbb{P}_{#1} \left( {#2} \right)}}
\newcommand{\Pkern}[1][]{\ensuremath{{}^{ #1}\!P}}

\newcommand{\Pfs}{\ensuremath{\mathbb{P}\text{-a.s.}}}
\newcommand{\Erw}[2][]{\ensuremath{\mathbb{E}_{#1} \left( {#2} \right)}}

\newcommand{\N}{\mathbb{N}}
\newcommand{\Z}{\mathbb{Z}}

\newcommand{\R}{\mathbb{R}}

\newcommand{\C}{\mathbb{C}}

\newcommand{\Var}{Var}

\renewcommand{\S}{\mathcal{S}}

\newcommand{\supp}{\mathrm{supp}}

\renewcommand{\epsilon}{\varepsilon}
\renewcommand{\rho}{\varrho}

\newcommand{\esl}[1]{\ensuremath{\left( #1 \right)^\sim}}

\newcommand{\1}[1][]{\mathbf{1}_{#1}}
\newcommand{\norm}[1]{\ensuremath{\left\| {#1} \right\|}}

\newcommand{\abs}[1]{\ensuremath{\left| {#1} \right|}}
\newcommand{\skalar}[1]{\langle #1 \rangle}

\newcommand{\eqdist}{\stackrel{d}{=}}

\newcommand{\Id}{\mathrm{Id}}

\newcommand{\mc}[1]{\mathcal{#1}}
\newcommand{\mf}[1]{\mathfrak{#1}}
\newcommand{\M}{\mf{M}}

\renewcommand{\S}{\mathcal{S}}

\begin{document}
\title[Tail behaviour of solutions of fixed
point equations]{Heavy tailed solutions of multivariate smoothing transforms}
\author[D. Buraczewski et al.]{Dariusz Buraczewski, Ewa Damek, Sebastian Mentemeier, Mariusz Mirek}

\thanks{{D.~Buraczewski and E.~Damek were partially supported by NCN grant
DEC-2012/05/B/ST1/00692.  S.~Mentemeier was supported by the Deutsche Forschungsgemeinschaft (SFB 878). M.~Mirek was  partially supported by MNiSW
grant  N N201 393937.}}

\maketitle

\begin{abstract} \ \\
Let $N > 1$ be a fixed integer and $(C_1, \dots, C_N,Q)$ a random element of
\makebox{ $M(d \times d, \R)^N \times \R^d$.} 
We consider solutions of multivariate smoothing transforms, i.e. random variables $R$ satisfying
$$R \eqdist \sum_{i=1}^N C_i R_i +Q $$
where $\eqdist$ denotes equality in distribution, and $R, R_1, \dots, R_N$ are independent identically distributed $\R^d$-valued random variables, and independent of $(C_1, \dots, C_N, Q)$. We briefly review conditions for the existence of solutions, and then study their asymptotic behaviour. We show that under natural conditions, these solutions exhibit heavy tails. Our results also cover the case of complex valued weights $(C_1, \dots, C_N)$.
\end{abstract}

\section{Introduction}



Let $N > 1$ be a fixed integer and $(C_1, \dots, C_N,Q)$ a random element of
{ $M(d \times d, \R)^N \times \R^d$.} This induces a mapping $\S : \M(\R^d) \to \M(\R^d)$
on the space of probability measures on $\R^d$, defined by
\begin{equation}\label{S}
\S \mu := \mathcal{L}\left( \sum_{i=1}^N C_i X_i +Q \right),
\end{equation}
where $X_1, \dots, X_N$ are independent identically distributed (iid) with
distribution $\mu$, independent of the \emph{weights and the immigration term} $(C_1, \dots, C_N, Q)$, and $\mathcal{L}$ denotes the law of $\sum_{i=1}^N C_i X_i + Q$.

 On a suitably chosen complete metric subspaces of
$\M(\R^d)$, $\S$ possesses  a unique fixed point.  In terms of
random variables, this means
\begin{equation}\label{SFPE}
R \eqdist \sum_{i=1}^N C_i R_i +Q,
\end{equation}
with $R, R_1, \dots, R_N$ iid and independent of $(C_1, \dots,
C_N,Q)$. Uniqueness of $R$ then means of course uniqueness of its
distribution. We will slightly abuse notation, and call the random
variable $R$ above a solution of the fixed point equation. In
fact, its distribution is a fixed point of $\S$.

We will also consider a particular case of \eqref{SFPE} when $Q=0$, i.e. we will study solutions of the equation
\begin{equation}\label{SFPE'}
R \eqdist \sum_{i=1}^N C_i R_i.
\end{equation}
It turns out that there are some subtle differences between those two cases. To distinguish between them we will call the stochastic equation \eqref{SFPE} inhomogeneous and \eqref{SFPE'} homogeneous.

\medskip

The case of nonnegative scalar weights, i.e. $C_i \in \R_+$, known
as the smoothing transform, has drawn much attention, see e.g. the
classical works \cite{DL1983,KP1976,Liu2000}, as well as
\cite{AM2010a} and the references therein. Also the case of
real-valued scalar weights has been studied, see \cite{CR2003} or
very recently in \cite{Meiners2012}.

The study of its multivariate analogue has a much shorter history,
but it draws more and more attention. We will focus on two
different kinds of assumptions, the first being $C_i$ from the set
of similarities, i.e.  products of orthogonal matrices and
dilations. This particularly covers the case of the smoothing
transform with complex valued weights. An equation of this type
was very recently studied in \cite{Chauvin2011}. Another
situation, where similarities appear as weights, is the joint
distribution of two statistics appearing in phylogenetic trees,
see \cite[Equation (14)]{Blum2006}; or the joint distribution of
key comparisons and key exchanges for Quicksort, see \cite[Theorem
4.1]{NR2006}. In all these  papers, only the case of
solutions with exponential moments has been studied. In particular
in the light of the article \cite{JO2010b}, it is tempting to
search for solutions with finite expectation, but heavy tails.

Secondly, we will treat { general matrices}, under some density assumptions on their distribution, as introduced in \cite{AM2010}. There, new properties, unknown from the one-dimensional case as well as prior multidimensional studies concerning positive solutions (\cite{BDG2011,Mirek2011a}), appear. We will discuss them by a nice example, which also allows us to (partially) answer a question raised by Neininger and R\"uschendorff \cite[Problem 3.2]{NR2006}.

{ A particular motivation for our work are equilibrium distributions of kinetic models. They can be studied as fixed points of smoothing transforms, as discussed for the one-dimensional case in the very recent article \cite{Bassetti2012+}. It is stated there \cite[Theorem 2.2 (iii, iv)]{Bassetti2012+} that these distributions show heavy tail behaviour in the sense that moments of high order are infinite. Our results (for similarities) then give precise tail asymptotics for these models. We will describe the basic ideas in subsection \ref{sect:example_mathes}, and derive heuristically a stochastic fixed point equation for the distribution of particle velocity in Maxwell gas. Naturally, this is a distribution on $\R^d$, and it was studied in \cite[Example 6.1]{Bassetti2011}, where it was shown to have heavy tails, \cite[6.1.4]{Bassetti2011}, but without precise tail index. Our results in the case of general matrices give these tail asympotics. The details are worked out in subsection \ref{sect:example_mathes}.    }

\subsection{What can be expected from the one-dimensional case?}

The main properties of the fixed points are governed by the
function $\hat{m}(s):= \E \sum_{i=1}^N C_i^s$, which is convex with
$\hat{m}(0)=N >1$, thus there are at most two solutions $0< \alpha < \beta$ to $\hat{m}(s)=1$.
It is shown in \cite[Theorem 8.1]{AM2010a}, that the solutions to \eqref{SFPE}
are of the form $$ X = R + hW^{1/\alpha}Y, $$
where $h \ge 0$  and $Y$ is a one-sided stable law
of index $\alpha$ independent of $(R,W)$, if $\alpha \in (0,1)$, and $Y=1$ if $\alpha=1$. The random
variables $(R,W)$ can be expressed in terms of the weighted branching process (WBP)
associated with $(C_1, \dots, C_N,Q)$. The WBP will be defined in Section \ref{sect:WBP}.
Note that in the recent preprint \cite{Meiners2012}, there is a similar result for real-valued weights, and vector valued solutions. This corresponds to diagonal matrices (which are of course similarities) in our setting.

If $h >0$, tails of the solutions are governed by $\alpha$. If $h=0$, then the
solution is given by $R$, and it is this special solution (called minimal
solution in \cite{AM2010a}), which we are interested in. It was first shown
by Guivarc'h \cite{Gui1990} (for the homogeneous case \eqref{SFPE'} and $\a=1$) that
given the existence of $\beta$, the solution $R$ has heavy tails with index
$\beta$, i.e.
\begin{equation}\label{heavy tails}
\lim_{t \to \infty} t^\beta \P{R > t} = K.
\end{equation} For the inhomogeneous equation, and also for the case
of real-valued weights, this was recently shown by { Jelenkovi{\'c}} and
Olvera-Cravioto \cite{JO2010a,JO2010b} (see also \cite{BDZ,BK}).

We extend their work to the multivariate case, and give in particular a nice
holomorphic argument that shows $K>0$, at least for similarities, in the general case it gives an equivalent condition.
The main part of this work is devoted to the study of the case of general regular matrices; most of the  proofs carry over
to the case of similarities, which will be treated separately.

\subsection{Acknowledgements}
The authors want to thank Gerold Alsmeyer and Jacek Zienkiewicz for helpful discussions during the preparation of the paper, and Daniel Matthes for information about kinetic models and pointing out reference \cite{Bassetti2011}. We are grateful to the referee for a very careful reading of the manuscript and many helpful suggestions improving the presentation.

\section{Statement of results}

\subsection{Notation and assumptions}
The Euclidean space $\R^d$ is endowed with the scalar product $\langle x,
y\rangle=\sum_{i=1}^{d}x_{i}y_{i}$ and the norm $|x|=\sqrt{\langle x, x\rangle}$. The unit sphere in $\R^d$ is denoted by $S$, and the projection of
a vector $x\in\R^d \setminus\{0\}$ on $S$ by $\esl{x}:= \abs{x}^{-1} x$.
Moreover,
on the space { $M(d\times d, \R)$ of $d \times
d$- real matrices } we will consider the operator norm $\|\cdot\|$ associated with the Euclidean norm $|\cdot|$
on $\R^d$, i.e. $\|a\|=\sup_{x\in S}|ax|$ for every $d \times
d$ matrix $a$.  The open ball of radius $\delta$ around a matrix $A$ w.r.t. to the operator norm is denoted by $B_\delta(A).$ The Lebesgue measure on $M(d \times d, \R)$, identified with $\R^{d^2}$, is denoted by $\lambda^{d \times d}$. We abbreviate $\R^+=(0, \infty)$.

The set of continuous functions $f : E \to \R$ on a metric space $E$ is denoted by $C(E)$, and by $C_C(E)$ we denote the compactly supported continuous functions $f : E \to \R$. The set of $k$-times (Frech\'et) differentiable functions is denoted by $C^k(E)$.

Let $I$ be a uniformly distributed random variable on the set $\{1, \dots, N\}$, independent of the random variable
$(C_1, \dots, C_N,Q)$. We introduce the random variable $C \eqdist C_I$, together
with a sequence $(C^{(n)})_{n \ge 0}$ of iid copies of $C$. Products of $C^{(1)},\ldots, C^{(n)}$ will be denoted by $\Pi_n := C^{(1)} \cdot\ldots\cdot C^{(n)}$.
We will see that the right multivariate expression for $\hat{m}$ is given by the
function
\begin{equation}\label{Def:m}
m(s) := N \lim_{n \to \infty} \left( \E \norm{\Pi_n}^s
\right)^\frac1n ,
\end{equation}
which is defined for $0 < s < s_\infty := \sup \{ s >0 \ : \ \E \norm{C}^s <
\infty \}$. Note that $s_\infty = \sup\{ s >0 \ : \ \max \E \norm{C_i}^s <
\infty \}.$ Again, $m$ is a convex function, and we define
\begin{align}
& \alpha := \inf \{ s > 0 : m(s) \le 1 \} \\
& \beta := \sup \{ s > 0 : m(s) \le 1 \} .
\end{align}
We will always assume that $\alpha < \beta < s_\infty$, then $m(\alpha)=m(\beta)=1$.{Though interesting examples with $\beta=s_\infty$ exist, in applications, $m$ usually has to be approximated by simulations, so it is convenient to introduce the standing assumption $\beta < s_\infty$ for it will also be needed in parts of the proofs.  }
Later on, we will see that $m$ is differentiable (at $\beta$), then $m'(\beta) >0$.

{ Applying the random variable $C_I$ one can easily prove (Proposition \ref{prop:permutation}) that \textbf{ $C_1,\ldots,C_N$ are identically distributed (but dependent), which we  w.l.o.g. assume from now on}. }

It may happen that the distribution of $R$ is degenerate. We  exclude this
case by the assumption
\begin{equation}
\text{For all } r \in \R^d, \quad \P{r = \sum_{k=1}^N C_k r + Q }<1. \label{no
trivial solution} \tag{\textsf{not triviality}}
\end{equation}

As soon as the first moments of $R$, $(C_1, \dots, C_N,Q)$ exist \eqref{SFPE}
yields an identity for the expectation of $R$. In some cases, we have to \emph{assume}
that this identity indeed has a solution, i.e. there is $r \in \R^d$ such that
\begin{equation}
 r = N \E C r + \E Q .
\label{ev}
\tag{\textsf{eigenvalue}}
\end{equation}
{ Moreover, if second moments exists, the covariance matrix $\Sigma$ of a solution $R$ to the homogeneous equation \eqref{SFPE'} has to satisfy
\begin{equation}
\label{covariance} \tag{\textsf{covariance}}\Sigma = N\ \E\big[ C \Sigma  C^\top \big].
\end{equation} }

\subsection{Existence and uniqueness results}
The following existence and uniqueness results are obtained from results of \cite{NR2004}. For the reader's convenience we give some ideas in Appendix \ref{sect:EU}.

\begin{proposition}[homogeneous case]\label{EU:hom}
Let $(C_1, \dots, C_N)$ be a random element of { $M(d \times d, \R)^N$.}
Assume { $\alpha <\min\{2, s_\infty\}$ and let $r \in \R^d$, satisfy \eqref{ev}.
\begin{itemize}
\item Case $\alpha <2$ : Then there is a unique solution $R$ to the homogeneous equation \eqref{SFPE'} with
$\E R = r$, and $\E \abs{R}^s < \infty$ for all $s < \min\{s_\infty, \beta\}$. $R\equiv 0$ if and only if $r=0$ .
\item Case $\alpha =2$: Assume in addition that there is a symmetric and positive definite matrix $\Sigma \in M(d \times d, \R)$, satisfying \eqref{covariance}. Then there is a unique solution $R$ to the homogeneous equation \eqref{SFPE'} with $\E R =r$, covariance matrix $\Cov(R)=\Sigma$ and $\E \abs{R}^s < \infty$ for all $s < \min\{s_\infty, \beta\}$.
\end{itemize} }
\end{proposition}

\begin{proposition}[inhomogeneous case]\label{EU:inhom}
Let $(C_1, \dots, C_N,Q)$ be a random element of { $M(d \times d, \R)^N \times \R^d$.} Assume $0 < \alpha < \min\{s_\infty, 2\}$,
and, if $\alpha \ge 1$, let $r \in \R^d$ satisfy \eqref{ev}.
If $\E |Q|^s <\infty$ for some $\alpha <s< \min\{s_\infty, 2\}$ then
 there is a unique solution $R$ to the inhomogeneous equation \eqref{SFPE} with $\E R = r$.

 Moreover if $\E |Q|^{\min\{s_\infty, \beta\}}<\infty$, then  $\E \abs{R}^s < \infty$ for all $s < \min\{s_\infty, \beta\}$.
\end{proposition}

{
\begin{remark}
In dimension $d=1$, exploiting the identity \eqref{covariance} for the variance of a solution $R$, one can show that if $\alpha >2$, then there is no nontrivial solution with finite moment of order $s > \alpha$. See \cite{ADM2012} for a detailed discussion. Without further assumptions, this is not true in dimension $d \ge 2$, as the following (even deterministic) example shows: Let
\begin{equation}\label{example_alpha} C_1 = \dots = C_N = 
\begin{pmatrix}
N^{-\frac13} & 0 \\
0 & N^{-\frac12} \\
\end{pmatrix} .
\end{equation}
Then $m(s) = N \norm{C_1}^s = N (N^{-\frac13} )^s = N^{1-\frac{s}{3}}$, thus $\alpha =3$. But if $R_2$ has standard normal distribution, the random vector $(0, R_2)^\top$ is obviously a solution of the smoothing transform associated with \eqref{example_alpha}. The point is that $m$ is only concerned with the largest eigenvalue of $C$, but there may be solutions concentrated on subspaces. 
\end{remark} }

{ These existence and uniqueness results hold for general $d \times d$-matrices $(C_1, \dots, C_N) \in M(d \times d, \R)$. In order to describe the asymptotic behavior of $R$, we will also consider more specialised classes of matrices, namely invertible matrices, i.e. $(C_1, \dots, C_N) \in GL(d, \R)$, and the subclass of similarities.

The limit theorems are now stated separately for two main cases: matrices satisfying some irreducible and density hypotheses (here two subcases appear), and matrices  being similarities.}

\subsection{Asymptotic behavior - general matrices}

\medskip

We start with the first case studied recently in \cite{AM2010}. We will assume that $C$ acts irreducibly on the sphere, i.e.
\begin{equation}
\label{irred} \tag{\textsf{irred}} \forall _{x \in S} \ \forall _{ \text{ open } U \subset S} \quad \quad \max_{n \ge 1} \P{\esl{x\Pi_n} \in U}>0,
\end{equation}
and that the law of $C$ is spread out, i.e. 
\begin{equation}
\label{density} \tag{\textsf{density}} \exists _{A \in GL(d,\R)} \ \exists _{c,\delta >0} \  \exists  _{n \in \N}
\quad \quad \P{\Pi_n \in \cdot} \ge c \1[B_\delta(A)] \lambda^{d\times d} .
\end{equation}

\begin{theorem}\label{main theorem}
Let $(C_1, \dots, C_N,Q)$ be a random element of  { $GL(d,\R)^N \times \R^d$},
and $R$ be the unique solution to \eqref{SFPE} given by Propositions
\ref{EU:hom}, resp. \ref{EU:inhom}. Let the assumptions of these propositions
hold, and assume the existence of $\alpha < \beta < s_\infty$ such that $m(\alpha) = m(\beta)=1$, as well as {  $\E \abs{Q}^{s_\infty} < \infty$. } Let the
conditions \eqref{irred},\eqref{density} hold.

Then
\begin{equation}
\lim_{t \to \infty} t^\beta \P{xR >t} = K \cdot e(x),
\end{equation}
for a positive continuous function $e:S\mapsto(0, \8)$. $K >0$ if and only if $\E \abs{R}^\beta = \infty$.
\end{theorem}

In the light of \cite{BDM2002,BL2009}, this property is (as long as $\beta$ is not an even integer) equivalent to multivariate regular variation: There is a unique Radon measure $\Lambda$ on $\R^d$, such that for every $f \in C_C(\R^d\setminus\{0\})$, the set of compactly supported functions,
$$ \lim_{t \to \infty} t^\beta \E f(t^{-1}R) = \int_{\R^d\setminus\{0\}} f(x) \Lambda(dx).$$

{ Particularly motivated by \cite[Example 6.1]{Bassetti2011}, we additionally state the theorem with a slightly different set of assumptions, which are in particular suitable for random matrices of the type $C=YY^\top$, where $Y$ is a random vector with a spread-out density.

So here we do not  assume $C \in GL(d, \R)$, instead we need:
\begin{equation}
\label{notvanish} \tag{\textsf{not vanishing}} \forall_{x \in S} \P{xC = 0}=0 .
\end{equation}
{ A probability measure on $\R$ is called nonarithmetic, if the smallest closed  group containing its support is $\R$.} The densitiy assumption now reads as follows:
\begin{equation}
\label{minorization} \tag{\textsf{minorization}} \exists_{p >0} \ \exists_{\phi \in \M(S)} \ \exists_{\psi \in \M(\R), \psi \text{ nonarithmetic}}\ \forall_{x \in S} \quad \P{ \esl{xC} \in \cdot, \log \abs{xC} \in \cdot} \ge p\ \phi \otimes \psi .
\end{equation}

\begin{theorem}\label{main theorem2}
Let $(C_1, \dots, C_N,Q)$ be a random element of  $M(d \times d,\R)^N \times \R^d$,
and $R$ be the unique solution to \eqref{SFPE} given by Propositions
\ref{EU:hom}, resp. \ref{EU:inhom}. Let the assumptions of these propositions
hold, and assume the existence of $\alpha < \beta < s_\infty$ such that $m(\alpha) = m(\beta)=1$, as well as { $\E \abs{Q}^{s_\infty} <\infty$}. Let the
conditions \eqref{notvanish},\eqref{minorization} hold.

Then
\begin{equation}
\lim_{t \to \infty} t^\beta \P{xR >t} = K \cdot e(x),
\end{equation}
for a positive continuous function $e:S\mapsto(0, \8)$. { Moreover, }$K >0$ if and only if $\E \abs{R}^\beta = \infty$.
\end{theorem}

}

\subsection{Asymptotic behavior - group of similarities}

Next we will assume that $C_i$ are elements of the group of similarities
i.e. the group of elements $g$ of ${\rm GL}(d,\R)$ satisfying
$$
|gx| = \|g\||x|
$$ for every $x\in \R^d$. The group of similarities $G$ is the direct product of { the multiplicative group $\R^+$} and the orthogonal group $O(d)$.

\medskip

Notice that if $g\in G$ then its norm is given by its radial part, i.e. if $g=tk$, $t\in\R^+$, $k\in O(d)$, then $\|g\|=t$. This implies in particular
 that in this case $m(s) = N\E\|C\|^s$. We denote by $\mu$ the law of $C$ and by
 $G_{\mu}$ the subgroup of $G$ generated by the support of $\mu$. We will assume
 that { $\log \norm{C}$ is nonarithmetic.}
Then there is a subgroup $H_{\mu}$ of $O(d)$ such that{
$$ G_{\mu }=  \R ^+ \times H_{\mu}.$$}
(see \cite{BDGHU2009} for more about the structure of the group
$G$).

\begin{theorem}
\label{thm: similarities}
{ Let $(C_1, \dots, C_N,Q)$ be a random element of  $G^N \times \R^d$,
and $R$ be the unique solution to \eqref{SFPE} given by Propositions
\ref{EU:hom}, resp. \ref{EU:inhom}. Let the assumptions of these propositions
hold, and assume the existence of $\alpha < \beta < s_\infty$ such that $m(\alpha) = m(\beta)=1$. Assume moreover { that $\E|Q|^{\beta}<\8$ and that $\log \norm{C_i}$ }are nonarithmetic.} Then
$$
\lim_{t\to\8} t^\b \P{|R|>t} = K_+.
$$
Moreover there exists a unique Radon measure $\Lambda$ such that for any $f\in C_C(\R^d\setminus\{0\})$
$$
\lim_{\norm{a}\to 0, a\in G_{\mu}} \norm{a}^{-\b}\E f(aR) = \int_{\R^d\setminus\{0\}} f(x) \Lambda(dx).
$$
The measure $\Lambda$ is homogeneous and $\Lambda(dg) = \sigma(dk)\frac{dt}{t^{1+\b}}$ for some finite measure on $S$ such that
$$
\sigma(S) = \frac 1{m_\b} \E\bigg[ \bigg|\sum_{i=1}^N C_iR_i + Q\bigg|^\b    -\sum_{i=1}^N |C_iR_i|^\b  \bigg],
$$ where $m_\b = \E[\norm{C}^\b \log\norm{C}]>0$.
\end{theorem}

\begin{rem} The theorem stated above is an analogy with the one
obtained in \cite{BDGHU2009} for the solution of $R\stackrel{d}= CR+Q$.
\end{rem}

In this case of similarities, we obtain a much stronger dichotomy
concerning nontriviality of the limit measure: 
\begin{prop}\label{prop:dichotomy}
Suppose that the assumptions of Theorem \ref{thm:
similarities} are satisfied { and additionally $\E |Q|^{s_\infty}<\infty $.Then either $K_+$ and
$\sigma(S)$ are positive, or $\E \abs{R}^s < \infty$ for all $s <
s_\infty$. }
\end{prop}

From this we will deduce a sufficient condition for $\sigma(S)>0$:
\begin{proposition}\label{prop:positivity}
{ Under the assumptions of Theorem \ref{thm: similarities}, let  $\gamma\in(\beta, s_\infty)$, be such that $\E \norm{C}^\gamma =1$, and $\E \abs{Q}^{s_\infty} < \infty$. Then
\begin{equation}
\sigma(S)>0 \ \Leftrightarrow \  \text{\eqref{no trivial solution} holds.}
\end{equation} }
\end{proposition}

\subsection{The structure of the paper}

The organization of the paper is as follows. We start with a brief section introducing the weighted branching process, a stochastic model that allows to study iterations of $\S$ in terms of random variables. 
The proof of existence and uniqueness of solutions to \eqref{SFPE} and \eqref{SFPE'} (Proposition \ref{EU:hom} and Proposition \ref{EU:inhom}) is classical and we present only main arguments, therefore we postpone the proof to Appendix \ref{sect:EU}. We consider separately general matrices and similarities. Although ideas are similar, the two proofs require slightly different techniques. All proofs are divided into two parts. First we prove the existence of the limit and then the positivity of the limiting constant. 
General matrices are treated in Sections \ref{sect:Ts} and \ref{sect:implicit renewal}. First we introduce the tools (transfer operators) in Section \ref{sect:Ts} and then in Section \ref{sect:implicit renewal} we formulate the renewal theorem and conclude existence of the limit. Asymptotic behavior of solutions in the similarity case is considered  in Section \ref{sect:similarities}. Finally we finish proofs of our main results in Section \ref{sect:positivity}, where we prove positivity of the limiting constants. Since  the proofs consist of tedious calculations,
  for a better stream of arguments, some proofs have been carried over into  Appendix \ref{app:lemmas} and Appendix \ref{app:proofs}. In Section \ref{sect:example} we present applications of our results to very concrete models.

\section{The weighted branching
process}\label{sect:WBP}

In this section, we introduce a stochastic model, the weighted branching process (WBP), that produces a sequence of random variables $(Y_n)$ with $\mathcal{L}(Y_n)=\S^n(\mu)$.
We consider random variables, indexed by a $N$-ary tree { $\mathcal{T}$}. For a node
$v=(i_1, \dots, i_k)$, we denote its level by $\abs{v}=k$, its ancestor in the
$l$-th level, $l \le k$, by $v|l=(i_1, \dots, i_l)$ and its $i$th successor by
$vi=(i_1, \dots, i_k,i)$. The root is denoted by $\emptyset$.
To define the weighted branching process, we assign to each node $v$ a copy
$(C_1(v), \dots, C_N(v),Q)$ of the vector $(C_1, \dots, C_N,Q)$, independent
of all other random variables. The product along a path is { defined } recursively
by $L(\emptyset):=\mathrm{Id}$, the identity matrix, and
$ L(vi):= L(v)C_i(v)$.
Given a random variable $X$, we assign to each node also a copy $X(v)$ of $X$,
independent of all other random variables.
Then $$Y_n := \sum_{\abs{v}=n} L(v) X(v) + \sum_{k=0}^{n-1} \sum_{\abs{v}=k}
L(v) Q(v)$$ is called the weighted branching process associated with $X,(C_1,
\dots, C_N,Q)$. It is easy to see that if $X$ has distribution $\mu$, then $\S^n(\mu)=\mathcal{L}(Y_n)$.
Moreover, the weights $(L(v))_{\abs{v}=n}$ are dependent, but identically distributed with the same distribution as $\Pi_n$.

In many cases (see e.g. \cite{BDG2011, JO2010b, JO2010a, Mirek2011a})  one can prove that the sequence $Y_n$ converges pointwise and its limit provides a solution of the smoothing transform. However in order to prove existence theorems in full generality, one has to consider $\S$ as an operator on an appropriate complete metric space and apply the Banach fixed point theorem. The details are worked out in Appendix \ref{sect:EU}.

\section{Transfer operators and change of measure}\label{sect:Ts}
This section and the next section are devoted to the partial  proof of our main results, { Theorems \ref{main theorem} and \ref{main theorem2}}. 
We prove existence of the limit, and the positivity of the constant is postponed to Section \ref{sect:positivity}. 
Thus we assume that the $C_i$'s are in $GL(d, \R)$ and satisfy  \eqref{irred}, \eqref{density} or that the $C_i$'s are in $M(d \times d, \R)$ and satisfy \eqref{minorization} and \eqref{notvanish}. 

First, we study a family of transfer operators introduced by the action of $C$ on the sphere  $S$, and derive harmonic functions, which allow for a change of measure.
 The results of this section are mainly technical, however this is the main tool that will be used in the next section
 when applying the Markov Renewal Theorem. As a by-product, we prove holomorphicity of $m(s)$.

 \subsection{Transfer operators}

For $0\le s<s_\8$ we introduce the family of operators on continuous functions on the sphere, $T_s : C(S) \to C(S)$, defined by
$$ T_s f(x) := \E\big[f(\esl{x C} )\abs{x C }^s\big] .$$
{ They are well defined as mappings on $C(S)$ since $C \in GL(d, \R)$ resp. \eqref{notvanish} holds.}
One can easily check that powers of these  operators are given by 
$$ T_s^n f(x) := \E\big[f(\esl{x \Pi_n} )\abs{x \Pi_n }^s\big] .$$

Observe that $T_0$ is a Markov transition operator on $S$, with associated transition kernel defined by $P(x,A):=\Prob((x\Pi_{1})^{\sim}\in A)$ for $x\in S$ and measurable $A\subset S$. For compact subsets $D$ of $GL(d,\R)$, we further define the substochastic kernels $P_{D}(x,\cdot):=\Prob((x\Pi_{1})^{\sim}\in\cdot,\Pi_{1}\in D)$.
{ Then from assumptions \eqref{irred} and \eqref{density}, we may derive a property very similar to \eqref{minorization}:}

\begin{lemma}\label{corr:MC}
Suppose \eqref{irred} and \eqref{density}. There is $n \in \N$, $p>0$, a compact subset $D$ of $GL(d,\R)$ and a probability measure $\phi$ with $\supp (\phi) = S$ such that the minorization condition
\begin{equation}\label{MC} \tag{\textsf{MC}}
P^n(y, \cdot) \ge P_D^n(y, \cdot) \ge p \cdot \phi
\end{equation}
holds for all $y \in S$.
\end{lemma}

\begin{proof}
A weaker version of the lemma was proven in \cite{AM2010} (Lemma 2.1): {\it
 there exists for each $x\in S$ a compact subset
$D_x$ of $GL(d,\R)$ and $\delta_x,p_x>0$, $n_x\in\N$, and a probability measure $\phi_x$  with $\supp (\phi_x) =B_{\delta_x}(x)$ such that
\begin{equation}
P^{n_x}(y,\cdot) \ge P_{D_x}^{n_x}(y,\cdot) \ge p_x\phi_x
\end{equation}
for all $y\in S$. }

Therefore, for each $x \in S$ there is an open ball $B_{\delta_x}$ and a probability measure $\phi$ on it.
This is an open covering of the compact set $S$, choose a finite covering around points $x_1, \dots, x_k$.
Set $n:=\mathrm{LCM}(n_{x_1}, \dots, n_{x_k})$, and define
$$ L := \left\{ (\omega_1, \dots, \omega_n) \in \{0, n_{x_1}, \dots, n_{x_k}\}^n
\ : \ \sum_{i=1}^n \omega_i = n, \quad \omega_n \neq 0 \right\}.$$
For $\omega \in L$, set
$$D_\omega :=\{ M_1\cdot \ldots\cdot M_n \ : \ M_1 \in D_{x_{\omega_1}},  \dots, M_n
\in D_{x_{\omega_n}} \} ,$$ with $D_0$ containing only the identity matrix, and
$D := \bigcup_{\omega \in L} D_\omega$. Observe that this union is finite, so
$D$ is still compact. With $p_0 :=1$ define { $ p_\omega := \prod_{i=1}^n p_{x_i}^{\omega_i} >0$} and
$p :=\frac{1}{\abs{L}} \min_{\omega \in L}
p_\omega > 0$.
Finally, with $\phi:=k^{-1}( \phi_1 + \dots + \phi_k)$, we have
$$P^n_D(y, \cdot) \ge  \frac{1}{\abs{L}}  \sum_{\omega \in L}
P^{\omega_1}_{D_{x_{\omega_1}}} \otimes \ldots \otimes
P^{\omega_n}_{D_{x_{\omega_n}}}(y, \cdot) \\
\ge  \sum_{\omega \in L} \frac{1}{\abs{L}} p_\omega \phi_{x_{\omega_n}} \\
\ge  p \phi.$$
\end{proof}


\begin{remark} \label{rem:Doeblin}
It is a direct consequence of the existence and uniqueness results obtained in Appendix \ref{sect:EU}
  that we may as well study iterated versions of the fixed point equation, i.e.
$$ R \eqdist \sum_{\abs{v}=n}L(v) R(v) + \sum_{k=0}^{n-1} \sum_{\abs{v}=k} L(v)
Q(v) ,$$ where the matrices $(L(v))_{\abs{v}=n}$ are not independent, but
identically distributed, with the same distribution as $\Pi_n$, and { $(R(v))_{v \in \mathcal{T}}$ are iid with the same distribution as $R$ and independent of $(L(v),Q(v))_{v \in \mathcal{T}}$.} Thus, we may w.l.o.g. assume that { if \eqref{irred} and \eqref{density} hold, then }
\begin{quote}
\eqref{MC} holds with $n=1$
\end{quote}
- which we will do from now on.
\end{remark}

This result now allows us to deduce the following properties of $T_s$ from the results in \cite{AM2010}:
\begin{proposition}\label{prop:properties es nus}
{ Let conditions \{\eqref{irred}, \eqref{density}\} or \{\eqref{notvanish}, \eqref{minorization}\} hold,} and $0 \le s < s_\infty.$ Then the spectral radius and the dominant eigenvalue of $T_s$ are equal to $\kappa(s):=\frac1N m(s)$. There is a unique strictly positive continuous function $e_s:S\mapsto(0,\8)$ and a unique probability measure $\nu_s$ with $\supp(\nu_s)=S$ such that
$$ T_s \nu_s = \kappa(s) \nu_s, \quad T_s e_s = \kappa(s) e_s, \quad \int_S e_s(x) \nu(dx)=1 .$$
Moreover, $e_s(x)=e_s(-x)$ for all $x \in S$.
\end{proposition}

\begin{proof}
By the minorization condition \eqref{MC} with $n=1$, there is a compact subset $D$, and $\phi, p$ such that
\begin{multline*}
T_s f(x) 
\ge  \Erw{f(\esl{xC} \abs{xC}^s) \1[D](C)} \\
\ge  \min \{\abs{xg}^s \ : x \in S, g \in D, s \in I\} \cdot \Erw{f(\esl{xC} \1[D](C)} \\
=  c \int f(y) \Pkern_D(x,dy)
\ge  cp \int_{S} f(y) \phi(dy) .
\end{multline*}
{ The same follows from \eqref{minorization}.}
This shows that the operator $T_s$ is strictly positive, i.e. it maps a nonzero nonnegative continuous function to a strictly positive function on $S$. Now \cite[Lemma 5.2]{AM2010} states the existence of a dominant eigenvalue $\kappa(s)$, equal to the spectral radius; as well as the existence of the corresponding eigenmeasure. In \cite[Lemma 5.3]{AM2010}, it is shown that $\kappa(s)=\frac{1}{N}m(s)$. Properties of the eigenfunction $e_s$ are shown in \cite[Lemma 5.4]{AM2010}. The assertion about the support of $\nu_s$ is a direct consequence of the above calculation, since for all $f \in C(S)$, $f \ge 0, f \neq 0$,
$$  \kappa(s) \int_S f(x) \nu_s(dx) = \int_S (T_s f)(x) \nu_s(dx) \ge cp \int_S f(y) \phi(dy) >0 ,$$
for $\supp(\phi)=S$.
\end{proof}

As an observation, we note that along the same lines, one can also obtain a uniform minorization for the eigenmeasures $\nu_s$:

\begin{lemma}\label{lemma:supp:nu}
Let $I$ be a compact subset of $[0, s_\infty)$. Then there is
$p>0$ and a probability measure $\phi$ with support $S$,
such that
$ \nu_s \ge p \cdot \phi$ for $s\in I$.
In particular, $\supp(\nu_s)=S$ for all $s \in [0, s_\infty)$.
\end{lemma}

\subsection{Markov Random Walks and Change of Measure}
A main ingredient in the proofs in \cite{JO2010b,JO2010a} was application of a
suitable renewal theorem. We will use the Markov renewal theorem of
\cite{Alsmeyer1997} (see Section \ref{sect:implicit renewal}), which deals with
\emph{Markov Random Walks}: Let $(X_n, U_n)_{n \ge 0}$ be a temporally homogeneous Markov chain on $S
\times \R$ such that
$$ \P{(X_{n+1}, U_{n+1}) \in A \times B | X_n, U_n} = \tilde{P}(X_n, A \times B) \quad
\mathrm{a.s.} $$
for all $n \ge 0$ and a transition kernel $\tilde{P}$ from $S$ to $S \times \R$. Then the associated sequence
$(X_n, V_n)_{n \ge 0}$ with $V_n = V_{n-1} + U_n$ is also a Markov chain and
called the Markov Random Walk with \emph{driving chain} $(X_n)_{n \ge 0}$.

We will study the Markov Random Walk defined by
$$ (X_n, V_n) := (\esl{X_0 \Pi_n}, \log \abs{X_0 \Pi_n}) \eqdist
(\esl{X_{n-1}C}, V_{n-1} + \log \abs{X_{n-1}C}) .$$ Its initial distribution is noted via $\Prob_x(X_0=x, V_0=0)=1$.

Notice that

\begin{equation}
\Pkern[\beta] f(x,t) := \frac{N}{e_\beta(x)e^{\beta t}} \Erw[x]{f(X_1,U_1+t)e_\beta(X_1)e^{\beta
(U_1+t)}}
\end{equation}
is a Markov transition kernel, since {
$$ \Pkern[\beta] \1 (x,t)= \frac{N}{e_\beta(x)e^{\beta t}} \Erw[x]{\1[{S\times \R}](X_1,U_1+t)e_\beta(X_1)e^{\beta(U_1 +t)}} = \frac{N}{e_\beta(x)} T_\beta e_\beta(x) = 1 .$$}
The associated
probability  measure on the path space, $^\beta\Prob$ and the related expectation $\E_x$  are defined by
\begin{align}
&^\beta\Erw[x]{f(X_0,V_0,X_1, V_1, \dots, X_n, V_n)}\nonumber \\
&\hspace{1cm}:= \frac{N^n}{e_\beta(x)} \E_{x}\Big(e_\beta(X_n)e^{\beta V_n} f
\left(X_0,V_0,X_1, V_1, \dots, X_n, V_n \right)\Big), \label{Def:transformiertesMass}
\end{align}
for all bounded continuous functions $f$ and all $n \geq 0$.

The transition operator of the driving chain $(X_n)_{n \ge 0}$ under $^\beta\Prob$ is then given by
\begin{equation}\label{eq:phat}
^\beta\hat{P}f(x) = \frac{N}{e_\beta(x)} (T_\beta e_\beta f)(x).
\end{equation} From the minorization condition \eqref{MC} with $n=1$, { resp. \eqref{minorization}} one obtains directly that it satisfies the Doeblin condition, i.e. for all $x \in S$
$$ ^\beta\hat{P}(x,\cdot) \ge c \cdot \phi( \cdot),$$
for the probability measure $\phi$ as defined in \eqref{MC}. Thus it is in particular positive Harris recurrent.

By the regeneration procedure of Athreya and Ney \cite{Athreya1978}, there is a
sequence of random times $(\sigma_n)_{n \ge 1}$, called regeneration epochs,
such that for each $k \ge 1$, $(X_{\sigma_k +n})_{n \ge 0}$ is independent of $(X_j)_{0 \le j \le \sigma_k
-1}$ with distribution $\P[\phi]{(X_n)_{n \ge 0} \in \cdot}$ (see
\cite[Lemma 4.1]{AM2010}), with $\phi$ defined in Corollary \ref{corr:MC} { resp. given by \eqref{minorization}}.

The following lemma is a consequence of the minorization in \eqref{MC} via the substochastic kernels $\Pkern_D$:

\begin{lemma}[{\cite[Lemma 5.6]{AM2010}}]\label{lem:bounded increments}
We can choose a sequence of regeneration epochs $(\sigma_n)_{n \ge 0}$, such
that there is a finite interval $I \subset \R$ with $U_{\sigma_n} \in I$ $\Pfs$
for all $n \ge 1$.
\end{lemma}

We then say that \emph{$V_n$ has bounded increments at regeneration epochs}. 

{
\begin{remark}
Under \eqref{minorization}, we even have the stronger property that $(X_{\sigma_k +n}, U_{\sigma_k +n})_{n \ge 0}$ are independent of $(X_j, U_j)_{0 \le j \le \sigma_k
-1}$ and have distribution $\P[\phi \otimes \psi]{(X_n, U_n)_{n \ge 0} \in \cdot}. $ This strong type of regeneration is studied in \cite{Athreya1978b}. Obviously, it implies that $V_n$ has bounded increments at regeneration epochs.
\end{remark} }

All in all, we have proven the following result, where the last assertion about $\pi$ is a direct consequence of the formula \eqref{eq:phat} for $^\beta\hat{P}$.

\begin{prop}\label{propertiesMRW}
Under the measure $^\beta\Prob$ as defined in \eqref{Def:transformiertesMass}, $(X_n, V_n)_{n \ge 0}$ is a Markov Random Walk, its driving chain $(X_n)$ is a Doeblin chain, thus positive Harris recurrent, and $V_n$ has bounded increments at regeneration epochs. The stationary distribution $\pi$ of $(X_n)$ is given by $$\pi(dx)=e_\beta(x) \nu_\beta(dx).$$
\end{prop}

\subsection{Properties of the operator $T_z$}

The final part of this section investigates further the mappings $s \mapsto \kappa(s)$ as well as $s \mapsto e_s$ and $s \mapsto \nu_s$. By means of a perturbation theorem for quasicompact operators, we will see that these mappings are holomorphic on some small ball around $\beta$. First, we show that $s \mapsto T_s$ { is in fact  a} holomorphic mapping.
For $0 < \Re z < s_\infty$, define the operator $T_z$ by
$$ T_z f(x) := \Erw{\abs{xC}^z f(\esl{xC})}.$$

A family of linear operators $T_z : B \to B$ is called \emph{weakly holomorphic}, if for every $x \in B$, $y \in B'$ { ($B'$ is the dual space of $B$)} the function {  $$ z \mapsto y(T_z x), $$ } is holomorphic. In our case $B=C(S)$, this is equivalent to being (strongly) holomorphic, see e.g. \cite[Exercise 8.E.]{Mujica1986}.

\begin{lemma}\label{Tz:holomorph}
The mapping $z \mapsto T_z$ is (strongly) holomorphic on the domain $0 < \Re z < s_\infty$.
\end{lemma}
\begin{proof} The proof that $T_z$ is holomorphic is taken from \cite[p. 253]{GL2004}: Let $\gamma$ be a closed path in the domain $0 <\Re z < s_\infty$, then for any bounded continuous function $f$ and finite measure $\nu$, we show that $ \int_\gamma \skalar{\nu, T_z f} dz =0$:
\begin{align*}
\int_\gamma \skalar{\nu, T_z f} dz = & \int_\gamma \int_S \int_{GL(d, \R)} f(\esl{Mx}) \norm{Mx}^z \P{C_1^\top \in dM} \nu(dx) dz \\
= & \int_S \int_{GL(d, \R)} f(\esl{Mx}) \left( \int_\gamma \norm{Mx}^z dz \right) \P{C_1^\top \in dM} \nu(dx),
\end{align*}
and the innermost integral is zero since $z \mapsto \norm{Mx}^z$ is holomorphic.

This, together with the fact that $z \mapsto \skalar{\nu, T_z f}$ is continuous, implies that $T_z$ is { weakly} holomorphic, thus already strongly holomorphic. \end{proof}

For a linear operator $T$ denote its spectral radius by $r(T)$. A linear operator $T$ on $C(S)$ is said to be \emph{quasi-compact}, if $C(S)$ can be decomposed into two $T$-invariant closed subspaces
$$ C(S) = F \oplus G,$$
where $r(T_{|G}) < r(T)$, while $\mathrm{dim} F < \infty$ and each eigenvalue of $Q_{|F}$ has modulus $r(T)$.

\begin{lemma}
The operator $T_\beta$ is quasi-compact.
\end{lemma}
\begin{proof}
With the results of Proposition \ref{prop:properties es nus}, what remains to show is the spectral gap property. But this follows, since $^\beta\hat{P}$ satisfies the Doeblin condition, and its spectral properties are one-to-one (resp. 1-to-$\frac1N$) with those of $T_\beta$, since for all $f \in C(S)$,
$$T_\beta f(x) = \frac{e_\beta}{N} {^\beta \hat{P}}(f/e_\beta)(x)  .$$
\end{proof}

Then we may apply the perturbation theorem \cite[Theorem III.8]{HH2001} and derive for our situation the following corollary:

\begin{corollary}\label{cor:perturbation}
There is $\delta >0$, such that for all $z \in B_\delta(\beta)$, $T_z$ has a simple dominating eigenvalue $\kappa(z)$, with eigenfunction $e_z$ and eigenmeasure $\nu_z$, and all mappings
$$ \kappa : B_\delta(\beta) \to \C, \quad  e_{\bullet} : B_\delta(\beta) \to C(S) \ \text{ and } \ \nu_{\bullet} : B_\delta(\beta) \to C(S)' $$
are holomorphic. In particular, $m=N\kappa$ is differentiable in $\beta$.
\end{corollary}

\section{Implicit Markov Renewal theory on trees} \label{sect:implicit renewal}

We will show the following proposition, which gives proof  of existence of the limit in Theorems \ref{main
theorem} and \ref{main theorem2}:

\begin{proposition}\label{prop:implicit renewal}Under the assumptions of Theorem
\ref{main theorem}, { resp. Theorem \ref{main theorem2}},
\begin{equation}\label{theorem:main:assertion1}
\lim_{t \to \infty} t^{\beta}\,\P{xR > t} = \frac{e_\beta(x)}{2 \beta
l_\beta}
\int_S
\Erw{\abs{\sum_{i=1}^N y C_i R_i +yQ}^\beta - \sum_{i=1}^N \abs{y C_i
R_i}^\beta} \nu_\beta(dy),
\end{equation}
for all $x \in S$, where
$$ l_\beta=\int_S \Erw{e_\beta(\esl{yC}) \abs{yC}^\beta \log \abs{yC} } \nu_\beta(dy)
>0 .$$
\end{proposition}

But first, we have to introduce the Markov Renewal Theorem, which we will use in the subsequent proof.

\subsection{Markov Renewal Theory}
A measurable function $g : S \times \R \to \R$ is called \emph{$\pi$-directly Riemann integrable} if
\begin{align}
&g(x, \cdot) \text{ is $\lambda$-a.e.\ continuous for $\pi$-almost all } x \in S  \label{dRi1}\\
\text{and}\quad&\int_S \sum_{n \in \Z} \sup_{t \in [n \delta, (n+1) \delta)} \abs{g(x,t)} \pi(dx) < \infty \ \text{ for some } \delta > 0, \label{dRi2}
\end{align}
where $\lambda$ denotes the Lebesgue measure on $\R$.
The following Markov renewal theorem (MRT) is the main result of
\cite{Alsmeyer1997}:

\begin{theorem}\label{MRT}
Let $(X_n, V_n)_{n \geq 0}$ be a nonarithmetic MRW with positive Harris
recurrent driving chain $(X_n)_{n \geq 0}$ with stationary distribution $\pi$. Let $l := \E_\pi {V_1} >0$. If $g : S \times \R \to \R$ is a $\pi$-directly Riemann integrable function, then for $\pi$-almost all $x \in S$,
\begin{equation}\label{MRT:limit 1}
g *\Bbb{U}_x(t) := \E_x \left( \sum_{n \geq 0} g(X_n, t - V_n) \right) \ _{\overrightarrow{t\to\infty}}\ \frac{1}{l} \int_S \int_{\R} g(u,v)\,dv\,\pi(du).
\end{equation}
\end{theorem}

\begin{remark}\label{crucial extension MRT}
The following extension of the above result is given in \cite[Section
7]{AM2010}: If $(X_n)_{n \ge 0}$ is a Doeblin chain with bounded increments at
regeneration epochs, then the assertion \eqref{MRT:limit 1} holds for all $x \in S$.
In our case, this is assured by Proposition \ref{propertiesMRW} { resp. the subsequent remark}.
\end{remark}


{ For the precise formulation of nonarithmeticity in the context of Markov Random Walks, see \cite{Shurenkov1984}.  We only note here that if \eqref{irred}, \eqref{density} hold,  an adaptation of \cite[Lemma 5.8]{AM2010} shows that $(X_n, V_n)_{n \ge 0}$ is
nonarithmetic under $^\beta\Prob_x$; and that it is a direct consequence of \eqref{minorization}. }
Using Proposition \eqref{propertiesMRW} and the definition of $^\beta\E$ in \eqref{Def:transformiertesMass}, we compute \begin{align*}^\beta\E_\pi V_1 = & \int_S e_\beta(y)^{-1}
  \Erw{e_\beta(\esl{yC}) \abs{yC}^\beta \log \abs{yC} } \pi(dy)\\
  = & \int_S \Erw{e_\beta(\esl{yC}) \abs{yC}^\beta \log \abs{yC} } \nu_\beta(dy)
=  l_\beta.
\end{align*}
Using $m'(\beta)>0$ and \cite[Lemma 5.9]{AM2010} yield that $l_\beta >0$.

\subsection{Implicit Markov Renewal Theory on Trees}

Now we give the proof of Proposition \ref{prop:implicit renewal}.
Let us define
$$
f(x,t) = \frac{e^{\b t}}{e_\b(x)} \P{xR>e^t}.
$$ We would like to write the function $f$ as a potential of some function $g$ and then to apply Theorem \ref{MRT}. However the function $f$ is not sufficiently smooth to satisfy all the hypotheses of the renewal theorem. Therefore we consider its smoothed version, i.e. for any function
$g : S \times \R \to \R$ we define its exponential smoothing 
 $$ \hat{g}(y,t) = \int_{-\infty}^t e^{- (t-s)} g(y,s) ds .$$
By \cite[Lemma 9.3]{Goldie1991}, 
 if one of $f(x,t)$ and $\hat{f}(x,t)$, converges for $t \to \infty$, then both of them
converge to the same limit. So it is sufficient to consider
the exponential smoothed version of $f$.

\begin{lem}
The function $\hat f$ satisfies
$
\hat f(x,t) = \hat g * \mathbb{U}_x(t),
$ where {
$$
g(x,t) = \frac{e^{\b t}}{e_\b(x)} \Big[ \P{xR > e^t} - N\P{xCR > e^t}  \Big].
$$}
\end{lem}
\begin{proof}
First
we expand $\P{xR > e^t}$ into a telescoping sum, using the WBP:
\begin{eqnarray*}
\P{xR>e^t} &= & \sum_{k=0}^{n-1} \left[ \sum_{\abs{v}=k} \left( \P{xL(v)R(v) > e^t}
- \sum_{i=1}^N \P{xL(vi)R(vi) > e^t} \right) \right]  \\
& & + \sum_{\abs{v}=n} \P{xL(v)R(v) >
e^t} \\
&= & \sum_{k=0}^{n-1} \left[ \sum_{\abs{v}=k} \left(\P{x\Pi_k R > e^t} -
\sum_{i=1}^N
\P{x\Pi_k CR > e^t} \right)  \right] + N^n \P{x\Pi_n R > e^t} \\
&= &
\sum_{k=0}^{n-1} N^k \left[  \P{X_k e^{V_k} R > e^t} - N \P{X_ke^{V_k} CR > e^t}
\right] + N^n \P{x\Pi_n R > e^t} \\
&= & \sum_{k=0}^{n-1} N^k \int
\P{y R > e^{t-v}} - N \P{y CR > e^{t-v}} \P[x]{X_k \in dy, V_k \in dv} \\ && +
N^n \P{x\Pi_n R > e^t}.
\end{eqnarray*}
Multiplying both sides by $e^{\beta t}/e_\beta(x)$ and using
\eqref{Def:transformiertesMass}, we obtain
\begin{eqnarray*}
f(x,t)&=  & \sum_{k=0}^{n-1}  \int \frac{e^{\beta (t-v)}}{e_\beta(y)} \left[ \P{y R >
  e^{t-v}} - N \P{y CR > e^{t-v}} \right] N^k { \frac{e_\beta(y)e^{\beta v}}{e_\beta(x)} } \P[x]{X_k
  \in dy, V_k \in dv} \\ && + \frac{e^{\beta t}}{e_\beta(x)}N^n \P{x\Pi_n R > e^t} \\ &=   & \sum_{k=0}^{n-1}  \int \frac{e^{\beta
  (t-v)}}{e_\beta(y)} \left[ \P{y R > e^{t-v}} - N \P{y CR > e^{t-v}} \right]  \
  ^\beta\P[x]{X_k \in dy, V_k \in dv} \\ && +  N^n \frac{e^{\beta
  t}}{e_\beta(x)} \P{x\Pi_n R > e^t} \\ &= &  \sum_{k=0}^{n-1}  \int
  g(y,t-v) \ ^\beta\P[x]{X_k \in dy, V_k \in dv} + N^n \frac{e^{\beta
  t}}{e_\beta(x)} \P{x\Pi_n R > e^t}.
\end{eqnarray*}

Applying to { both sides the exponential smoothing }we have
\begin{equation}
\label{implicit:renewal:1}
\hat f(x,t) =
  \sum_{k=0}^{n-1}  \int \hat{g}(y,t-v)  \  ^\beta\P[x]{X_k \in dy, V_k \in
 dv} +
 \int_{-\infty}^t e^{-(t-s)} \frac{e^{\beta s}}{e_\beta(x)} N^n \P{x\Pi_n R >
 e^s} ds.
\end{equation}
Now we want to pass with $n$ to infinity and prove that the second term vanishes.
To this purpose choose $ \delta >0$ and $n_0 \in \N$ such that
$N\left( \E \norm{\Pi_n}^{\beta -\delta} \right)^\frac{1}n < 1-\epsilon$ for all
$n \ge n_0$ and some $\epsilon >0$. This is possible since $m(\beta - \delta) <
1$. Then, using the Markov inequality,
\begin{multline*}
 \int_{-\infty}^t e^{-(t-s)} \frac{e^{\beta s}}{e_\beta(x)} N^n \P{x\Pi_n R > e^s} ds
\le  \int_{-\infty}^t \frac{e^{- t + (\beta + 1) s}}{e_\beta(x)} N^n
\P{\abs{\Pi_n R} > e^s} ds \\
\le    \int_{-\infty}^t \frac{e^{- t + (\beta +
1) s}}{e_\beta(x)} N^n \frac{\Erw{ \abs{\Pi_n R}^{\beta -\delta}}}{e^{s(\beta -
\delta)}} ds \le    e^{- t} \int_{-\infty}^t \frac{e^{ (\delta + 1) s}}{e_\beta(x)} (1-\epsilon)^n \Erw{ \abs{R}^{\beta -\delta}} ds  \\
\le  e^{\delta t}  \cdot (1- \epsilon)^n,
\end{multline*}
which tends to zero for each fixed $t$, as $n \to \infty$.

The first term in \eqref{implicit:renewal:1} for $n \to \infty$  is \emph{almost} a renewal function, we want to switch summation and integration, to read
\begin{equation}
\hat{g}*\mathbb{U}_x(t) := \int \hat{g}(y, t- v) \sum_{k=0}^\infty
\ ^\beta\P[x]{X_k \in dy, V_k \in dv} .
\end{equation}
Therefore it remains to prove  that $\hat g$ is a directly Riemann integrable function.

In view of \cite{Goldie1991} Lemma 9.2,  it is sufficient to show $
\sup_{y\in S} \int_\R |g(y,s)|\lambda(ds)<\8
$,
i.e.
(after a change of variables)
\begin{equation} \label{dRi:condition}
\sup_{y \in S} \int_0^\infty t^{\beta -1} \abs{\P{yR> t} - N\P{yTR> t}} dt < \infty.
\end{equation}
Finiteness of the expression above follows from Propositions \ref{prop:Goldie lemma} and \ref{extension of moments}. Since both Propositions are very technical we postpone their statements to Section \ref{sect:moments} and their proofs to Appendix \ref{app:proofs}.
 \end{proof}

\begin{proof}[Proof of Proposition \ref{prop:implicit renewal}]
The Markov Renewal Theorem \ref{MRT} yields
\begin{align*}
&  \lim_{t \to \infty} \frac{e^{- t}}{e_\beta(x)} \int_{-\infty}^t e^{(\beta+1)s}
\P{xR > e^s} ds
=  \frac{1}{l} \int_S \int_\R \hat{g}(u,v) \lambda(dv)
\pi(du) \\
= &  \frac{1}{l} \int_S \int_\R g(u,v) \lambda(dv) \pi(du)
=  \frac{1}l \int_S \int_\R \frac{e^{\beta
  v}}{e_\beta(u)} \left[ \P{u R > e^{v}} - N \P{u CR > e^{v}} \right]
  \lambda(dv) \pi(du) \\
= &  \frac{1}l \int_S \int_0^\infty w^{\beta-1} \left[ \P{u R > w} - N \P{u CR > w} \right]
  dw \frac{1}{e_\beta(u)}\pi(du) \\
  = &  \frac{1}{2l} \int_S \int_0^\infty w^{\beta-1} \left[ \P{\abs{u R} > w} - N \P{\abs{u CR} > w} \right]
  dw \frac{1}{e_\beta(u)}\pi(du) \\
= & \frac{1}{2\beta l}
\int_S
\Erw{{\left[\sum_{i=1}^N u C_i R_i +uQ\right]^+}^\beta - \sum_{i=1}^N {[u C_i
R_i]^+}^\beta} \nu_\beta(du),
\end{align*}
for all $x \in S$. The penultimate line is justified by the symmetry of $e_\beta$ (Proposition \ref{prop:properties es nus}), the last identity will again be justified by   Propositions \ref{prop:Goldie lemma} and \ref{extension of moments}.

With this identity, in view of \cite[Lemma 9.3]{Goldie1991} we may easily unsmooth  $\hat{f}$, and one can finally infer that
$$ \lim_{t \to \infty} e^{\beta t} \P{xR > e^t} = \frac{e_\beta(x)}{2\beta l}
\int_S
\Erw{\abs{\sum_{i=1}^N u C_i R_i +uQ}^\beta - \sum_{i=1}^N \abs{u C_i
R_i}^\beta} \nu_\beta(du).$$

\end{proof}

The remaining part of the proof of Theorem \ref{main theorem} and \ref{main theorem2}, i.e.
the proof of the positivity of the limiting constant $K$ is postponed  to the Section \ref{sect:positivity}.

\section{Similarities}\label{sect:similarities}

In the last two sections we considered general matrices and proved existence of the limits in Theorems \ref{main theorem} and \ref{main theorem2}. It still remains to prove nondegeneracy of those limits and the arguments will be given in Section \ref{sect:positivity}.

 Now we consider similarities and our aim is to prove  existence of the limit in Theorem \ref{thm: similarities}.
 We  assume now that the $C_i$ take their values in the similarities group and that the assumptions of Theorem \ref{thm: similarities} are satisfied. The idea of the proof resembles the previous case of general matrices, i.e. we reduce the problem to the renewal equation but this time we apply an extended version of the renewal theorem for random walks on $G =\R^+\times O(d)$ (Theorem A.1 \cite{BDGHU2009}).

We denote by $\ov \mu$ the law of $C$. Since $N\int_G\|g\|^\b\ov \mu(dg)=1$,  the measure $\mu_\b (dg) = N\|g\|^\b\ov\mu(dg)$ is a probability measure
and moreover $m_\b = \int_G \log\|g\|\mu_\b(dg) = N\E[\|C\|^\b\log\|C\|]>0$. By $\ov\mu^{*k}$ we denote the $k$th convolution power of $\ov\mu$, i.e. the law of $\Pi_k$. In the same way we introduce $\mu_\b^{*k}$. Let $U_\b$ denote the potential of $\mu_\b$, i.e.
$U_\b =\sum_{k=0}^{\8}\mu_\b^{*k}$. Then $U_\b$ is a Radon measure on $G$ and {  we have the following result:}
\begin{lem}
Let $f\in C_C(\R^d\setminus\{0\})$, then for any $a\in G$
$$
\|a\|^{-\b}\E\big[ f(aR)\big] = (\delta_a*U_\b)(\psi_f) = \sum_{k=0}^\8 \int_G \psi_f(ag)\mu_\b^{*k}(dg),
$$ where
$$
\psi_f(g) = \|g\|^{-\b} \E\bigg[ f(gR) - \sum_{i=1}^N f(g C_i R_i)
\bigg].
$$

\end{lem}

\begin{proof}
We write
\begin{multline*}
\|a\|^{-\b}\E \big[f(aR)\big] \\= \|a\|^{-\b} \sum_{k=0}^{n-1}\E\bigg[ \sum_{|v|=k} \bigg(f(a L(v) R(v)) - \sum_{i=1}^N f(a L(vi)R(vi))\bigg)
\bigg] + \|a\|^{-\b} \sum_{|v|=n} \E \big[f(a L(v)R)\big]\\
= \sum_{k=0}^{n-1} \|a\|^{-\b}N^k \E\bigg[ f(a\Pi_kR) - \sum_{i=1}^N f(a\Pi_k C_i R) \bigg]
+ \|a\|^{-\b}\sum_{|v|=n} \E\big[ f(a L(v)R)\big] \\
=\int _G\sum_{k=0}^{n-1} \|ag\|^{-\b} \E\bigg[ f(agR) -
\sum_{i=1}^N f(a g C_i R) \bigg] \|g\|^\b N^k \overline\mu^{*k}(dg)
+ \|a\|^{-\b} \sum_{|v|=n} \E \big[f(a L(v)R) \big]\\
= \sum_{k=0}^{n-1}\int_G \psi_f(ag)\mu_\b^{*k}(dg)+ \|a\|^{-\b} \sum_{|v|=n} \E\big[ f(a L(v)R)\big] ,
\end{multline*}
Notice that for $s$ such that $m(s) <1$, assuming ${\rm supp}f\cap B_{\eta}(0)=\emptyset$, we have
$$
\bigg| \sum_{|v|=n} \E \big[ f(a L(v)R)\big] \bigg| \le c N^n {\mathbb P}
\big[ |a\Pi_n R|>\eta \big]
\le \frac{N^n \E[\|\Pi_n\|^s|R|^s]}{(\eta \norm{a}^{-1})^s}
\le \frac{m(s) ^n \E|R|^s}{(\eta \norm{a}^{-1})^s}.
$$
Hence
$$
\lim_{n\to\8} \bigg| \sum_{|v|=n} \E\big[ f(a L(v)R)\big] \bigg|\|a\|^{-\b} = 0
$$
and so
$$
\|a\|^{-\b}\E\big[ f(aR)\big] = \sum_{k=0}^\8 \int_G  \psi_f(ag)\mu_\b^{*k}(dg).
$$

\end{proof}
To apply the renewal theorem we have to check that the function $\psi_f$ is directly Riemann integrable. We need the two following lemmas, whose proofs will be presented in Appendix \ref{app:lemmas}:
\begin{lem}\label{ew6.1}
Let $f\in C_C^2(\R^d\setminus \{0\})$ and let $n\in\N$ be fixed. Then there is a constant $c = c(f)$ such that for every $0<\eps \le 1$
and every $x_1,\ldots,x_n,q\in \R^d$
$$
\bigg| f\big( x_1+\ldots + x_n + q \big) - \sum_{j=1}^n f(x_j) \bigg| \le c \bigg(
|q|^{\eps} + \sum_{i\not=j} |x_i|^{\eps}|x_j|^{\eps}
\bigg)
$$
\end{lem}
\begin{lem}
\label{lemma:dri}
Let $R$ be a solution of \eqref{SFPE} such that for every $s<\b$, $\E|R|^s<\8$. Suppose that $N \E\|C_1\|^\b =1$, $\E |Q|^\b <\8$, then for any
 $f\in C_C^2(\R^d\setminus\{0\})$, the function $\psi_f$  is directly Riemann integrable  on $G = \R^+ \times K$.
\end{lem}
\begin{proof}[Proof of Theorem \ref{thm: similarities}]
By the renewal theorem (Theorem A.1, \cite{BDGHU2009}), for any $f\in C_C^2(\R^d\setminus\{0\})$
$$
\lim_{\norm{a}\to 0; a\in G_{\mu}}\norm{a}^{-\b} \E f(aR) = \frac  1{m_\b}\int_G \psi_f(g) dg,
$$ where $dg$ is the Haar measure on $G$ normalized in such a way that for any radial function $f$ on $G$: $\int_G f(d)dg = \int_{0}^\8 f(t)\frac{dt}t$.

\medskip

In the same way as in \cite{BDGHU2009} we show that
the convergence above is valid also for $f\in C_C(\R^d\setminus \{0\})$ and thus there exists a Radon measure $\Lambda$ on $\R^d\setminus\{0\}$
such that
$$
\lim_{\norm{a}\to 0, a\in G_{\mu}} \norm{a}^{-\b} \E f(aR) = \is f{\Lambda}.
$$
Since $\Lambda$ is homogeneous it can be written in the form  $\Lambda = \sigma\otimes\frac{dt}{t^{\b+1}}$, i.e.
$$\langle f, \Lambda\rangle:=\int_{0}^{\8}\int_Sf(tw)\sigma(dw)\frac{dt}{t^{\b+1}},$$
where $\sigma$ is a finite measure on $S$.

Finally we have to justify the formula for $\sigma$, i.e. to prove
$$
\sigma(S) = \frac 1{m_\b} \E\bigg[ \bigg| \sum_{i=1}^N C_i R_i +Q
\bigg|^\b - \sum_{i=1}^N |C_i R_i|^\b
\bigg].
$$
{ For this purpose take an arbitrary} radial function
 $f\in C_C(\R^d\setminus\{0\})$.
To simplify our notation define
$$
I_f(s) = \int_0^\8 f(t)\frac{dt}{t^{1+s}}.
$$
Then, on the one hand
$$
\is f{\Lambda} = \int_0^\8\int_{S} f(tw) \sigma(dw)\frac{dt}{t^{\b+1}}
 =\sigma(S) \int_0^\8 f(t) \frac{dt}{t^{\b+1}} = \sigma(S) I_f(\b),
$$
and on the other, since $\psi_f$ is also a radial function we have
$$
\is f{\Lambda}= \frac 1{m_\b} \int_0^{\8} \psi_f(t) \frac{dt}t.
$$
Thus
\begin{equation}
\label{e8}
\sigma(S) I_f(\b) = \frac 1{m_\b} \int_0^{\8} \psi_f(t) \frac{dt}t.
\end{equation}
Notice that for $s<\b$ we have
$$
\int_0^\8 t^{-s} \E f(t|R|)\frac{dt}t = \E|R|^s \int_0^\8 f(t) \frac{dt}{t^{1+s}} = \E|R|^s I_f(s).
$$
Hence for $s<\b$, since $\E|R|^s <\8$, we may write
\begin{eqnarray*}
\int_0^\8 t^{\b-s}\psi_f(t)\frac{dt}t &=& \E\bigg[ |R|^s - \sum_{i=1}^N \|C_i\|^s |R_i|^s \bigg]\cdot I_f(s)\\
&=& \E\bigg[ \bigg|\sum_{i=1}^N C_i R_i +Q \bigg|^s - \sum_{i=1}^N \|C_i\|^s |R_i|^s \bigg]\cdot I_f(s).
\end{eqnarray*}
Now, letting $s\to\b^-$, in view of \eqref{e8}, we obtain
\begin{eqnarray*}
 \sigma(S)I_f(\b) &=& \lim_{s\to\b^-}
\frac{1}{m_\b}\int_0^\8 t^{\b-s}\psi_f(t)\frac{dt}t\\
&=& \lim_{s\to\b^-}\frac 1{m_\b} \E\bigg[ \bigg|\sum_{i=1}^N C_i R_i +Q \bigg|^s - \sum_{i=1}^N \|C_i\|^s |R_i|^s \bigg]\cdot I_f(s)\\
&=& \frac 1{m_\b} \E\bigg[ \bigg|\sum_{i=1}^N C_i R_i +Q \bigg|^\b - \sum_{i=1}^N \|C_i\|^\b |R_i|^\b \bigg] I_f(\b).
\end{eqnarray*}

\end{proof}

\section{Positivity of K}\label{sect:positivity}

The aim of this section is to study positivity of the limiting constants. We have proven up to now that tails of solutions to equations \eqref{SFPE} and \eqref{SFPE'} behave regularly at infinity, but we still do not know whether the limit is non-degenerate. Now we fill this gap.  
  We start with some preliminary estimates, that will be used in the proofs. Next we consider similarities for which our results are much stronger and give a complete answer, i.e. we prove Propositions \ref{prop:dichotomy} and \ref{prop:positivity}. Finally we consider general matrices and complete the proofs of Theorems \ref{main theorem} and \ref{main theorem2}.

\subsection{Moment bounds}\label{sect:moments}
Before we pass to proofs of positivity of the limiting constants we formulate here two Propositions containing useful estimates. However, since their proofs are long and technical we postpone them to Appendix \ref{app:proofs}.

\begin{proposition} \label{prop:Goldie lemma}
Assume that the following expectations are finite:
\begin{align}
\sup_{y \in S} \E{\abs{\abs{\sum_{i=1}^N yC_i R_i +yQ}^s - \abs{\sum_{i=1}^N
yC_iR_i}^s}} \tag{E0(s)}, \label{E0} \\
\sup_{y \in S}  \E{\abs{\abs{\sum_{i=1}^N yC_i R_i}^s - \abs{\max_i yC_iR_i}^s}}
\tag{E1(s)}, \label{E1} \\
\sup_{y \in S}  \E\abs{\sum_{i=1}^N \abs{yC_iR_i}^s - \abs{\max_i yC_iR_i}^s}
\tag{E2(s)}. \label{E2}
\end{align}
Then for all $y \in S$, the identity
\begin{equation}\label{Goldie Lemma}
\Erw{\abs{\sum_{i=1}^N yC_i R_i +yQ}^z
- \sum_{i=1}^N \abs{yC_i R_i}^z} =z \int_{0}^{\infty} t^{z-1} \left(
\P{\abs{yR} > t} - N\P{\abs{yCR} >t} \right) dt .
\end{equation}
holds true (for complex $z = s+iv$). Both sides of the equation above are finite and
 define holomorphic functions in the infinite strip $0 < \Re z < s$ provided that $0<s<s_{\8}$.
Moreover
$$ \sup_{y \in S} \int_{0}^{\infty} t^{s-1} \abs{
\P{\abs{yR} > t} - N\P{\abs{yCR} >t}} dt \le \ref{E0} + \ref{E1} + \ref{E2}.$$
\end{proposition}

\begin{proposition}\label{extension of moments}
For $\rho < s_\infty$, there exists $\epsilon >0$, such that $\E
\abs{R}^{\rho-\epsilon} < \infty$ implies finiteness of $E0(\rho+\epsilon)$,
$E1(\rho+\epsilon)$, $E2(\rho+\epsilon)$.
\end{proposition}

\subsection{Similarities}
\begin{proof}[Proof of Proposition \ref{prop:dichotomy}]  
{ First note, that $K_+ >0$ if and only if $\sigma(S)>0$.}
We are going to use Propositions \ref{prop:Goldie lemma} and
\ref{extension of moments} in the formulation adjusted to
the similarities. Namely, let
\begin{align}
\sup_{y \in S} \E{\abs{\abs{\sum_{i=1}^N C_i R_i +Q}^s -
\abs{\sum_{i=1}^N
C_iR_i}^s}}, \tag{E0(s)'} \label{E0'} \\
\sup_{y \in S}  \E{\abs{\abs{\sum_{i=1}^N C_i R_i}^s - \max_i
\abs{C_iR_i}^s}},
\tag{E1(s)'} \label{E1'} \\
\sup_{y \in S}  \E\abs{\sum_{i=1}^N \abs{C_iR_i}^s - \max_i
\abs{C_iR_i}^s}.
\tag{E2(s)'} \label{E2'}
\end{align}
Then Proposition \ref{extension of moments} holds with the same
proof, and the finiteness of \eqref{E0'}, \eqref{E1'}, \eqref{E2'} implies, as before, the
identity
\begin{equation}\label{simlemma}
\Erw{\abs{\sum_{i=1}^N C_i R_i +Q}^z - \sum_{i=1}^N \abs{C_i
R_i}^z} =z \int_{0}^{\infty} t^{z-1} \left( \P{\abs{R} > t} -
N\P{\abs{CR} >t} \right) dt,
\end{equation}
for complex $z=s+iv$; moreover, both sides are holomorphic { for $\Re z < \beta + \epsilon$. 
On the other hand, we have for complex $z=s+iv$, $\alpha < s <\beta$, since $\E \abs{R}^s < \infty$,
$$
\E \bigg[ \abs{R}^z - \sum _{i=1}^N|C_iR_i|^z\bigg]= \E\bigg[ \abs{R}^z - \sum
_{i=1}^N|C_i|^z|R_i|^z\bigg]= (1-m(z))\E\abs{R}^z,$$
i.e.
$$
\E\abs{R}^z=\frac{\E \abs{R}^z - \sum
_{i=1}^N|C_iR_i|^z}{1-m(z)},\ \ \mbox{for\ \  $\Re z <\beta $.}$$
Now the rest
of the proof of Proposition \ref{prop:dichotomy} is the same as
the arguments in \cite{BDGHU2009}, but we include it here for
completeness.}

Suppose that $K=0$. Then both the numerator and the
denominator of the right hand side are holomorphic in $0<\Re
z<\beta +\eps $ and $1-m(z)$ has a simple zero at $\beta $
($m'(\beta)\neq 0$). Therefore, the right hand side is holomorphic
for $\Re z < \beta +\eps $. On the other hand, if $z=s\in \R$, the
left hand side is the Mellin transform $\hat \gamma (s)$ of the
law of $|R|$, which is well defined for $s<\theta _{\infty}$
called the abscissa of convergence of $\hat \gamma $. The Landau
theorem (see \cite{Widder1941}, Theorem 5, page 57) says that $\hat \gamma$
cannot be extended holomorphically to a neighborhood of $\theta
_\infty$. Hence $\beta + \eps < \theta _{\infty } ${,  i.e. $\E \abs{R}^{\beta + \epsilon}<\infty$. } Now
suppose that { $\theta _{\infty}<s_\infty $.} Then { \eqref{simlemma}} is
holomorphic for $\Re z < \theta _\infty +\eps$, and so repeating
the above argument (using that $m(z)\neq 0$ for $z\neq { \alpha}, \beta $)
we get a contradiction.
\end{proof}

\subsection{A sufficient condition for the limit to be positive}

{ We introduce the r.v.} $$ B:= \sum_{i=2}^N C_i R_i +Q,$$
since we will use some features of the affine equation $R \stackrel{d}= C_1R+B$ in the sequel.

\begin{lemma}
If \eqref{no trivial solution} holds, then $\P{C_1 r +
B=r}<1$, for all $r \in \R^d$.
\end{lemma}

\begin{proof}
Condition \eqref{no trivial solution} states that no Dirac measure solves the
fixed point equation, so $R$ takes at least two different values. If now
$$ r - C_1 r - Q = \sum_{i=2}^N C_i R_i ,$$
one sees by conditioning on $(C_1, \dots, C_N,Q)$ (then the LHS is constant),
that this identity can only hold on a set where all values of the $R_i$ are
fixed. But, since they are independent and non degenerated, this set has
probability smaller than one. So  $\P{C_1 r +
B=r}<1$ for all $r \in \R^d$ .
\end{proof}

\begin{proof}[Proof of Proposition \ref{prop:positivity}]
Obviously, if \eqref{no trivial solution} does not hold,  the solution is a constant, and $\sigma(S)=0$.
Conversely, assume that $\sigma(S)=0$, so by Proposition \ref{prop:dichotomy}, $\E \abs{R}^s<\infty$ for all $s < s_\infty$.
With the definition of $B$, $R$ satisfies the affine stochastic fixed point
equation $$ R \eqdist C_1 R +B ,$$ which was studied in \cite{BDGHU2009}. By the assumption $\E \abs{Q}^\gamma < \infty$, we also have $\E \abs{B}^\gamma < \infty$.
But then \eqref{no trivial solution} may not hold, since otherwise condition \textbf{H} of \cite{BDGHU2009} is satisfied, and it is shown in \cite[Proposition 2.6]{BDG2010}, that then $\E \abs{R}^\gamma = \infty$, which would be a contradiction.
\end{proof}

\subsection{General matrices}
In the case of general matrices, the crucial identity becomes more subtle. Using $T_s \nu_s = \frac{m(s)}{N} \nu_s$ and $\E \abs{R}^s < \infty$, $m(s)<1$ for all $s \in (\alpha,
\min\{\beta,s_\infty\})$ we have
$$ m(s)\int_S \E \abs{yR}^s \nu_s(dy) =  N \int_S T_s \E \abs{yR}^s \nu_s(dy) =
\sum_{i=1}^N \int_S \E \abs{yC_i R_i}^s \nu_s(dy) ,$$ and thus the
 identity
\begin{equation}
\label{identity} \int_S \E \abs{yR}^s \nu_s(dy) = (1 - m(s))^{-1} \int_S \Erw{
\abs{yR}^s - \sum_{i=1}^N \abs{yC_iR_i}^s } \nu_s(dy).
\end{equation}

We still can show that the right hand side has a holomorphic extension around $\beta$ if and only if $K=0$, but unfortunately, we cannot use the Landau lemma, because the LHS is not a Mellin transform, only a mixture of those.
Nevertheless, what we can show by this argument is: If $\E \abs{R}^\beta = \infty$, then the right hand side cannot have a holomorphic extension, thus $K >0$. (Note that $\E \abs{R}^\beta < \infty$ readily implies $K=0$.)

\begin{prop}\label{RHS:holomorphic}
There is $\delta >0$, such that
$$ K(z) := \int_S \Erw{
\abs{yR}^z - \sum_{i=1}^N \abs{yC_iR_i}^z } \nu_z(dy), $$
is a holomorphic function on $B_\delta(\beta)$.
\end{prop}

\begin{proof}
By Proposition \ref{extension of moments}, there is $\epsilon >0$ such that
$E0(\beta+\epsilon), E1(\beta+\epsilon), E2(\beta+\epsilon)$ are finite. By
Proposition \ref{prop:Goldie lemma} then for each $y \in S$,
$$ k(y,z):=\Erw{\abs{\sum_{i=1}^N yC_i R_i +yQ}^z
- \sum_{i=1}^N \abs{yC_i R_i}^z},$$
is holomorphic. By Corollary \ref{cor:perturbation},
there is a (possibly smaller) $\delta >0$, such that $\nu_z$ is a holomorphic
function on $B_\delta(\beta)$. But then
$$ K(z) = \int k(y,z) \nu_z(dy),$$
also defines a holomorphic function on $B_\delta(\beta)$.
\end{proof}

{ Also by Corollary \ref{cor:perturbation} , $m(s)=N \kappa(s)$} is holomorphic on $B_\delta(\beta)$, with $m(\beta)=1$. Since $m'(\beta)>0$ (by convexity), the right hand side in \eqref{identity} has a holomorphic extension around $\beta$ if and only if $K=0$.
Assume that $K=0$. The left hand side is a holomorphic function on $B_\delta(\beta)\cap\{\Re z < \beta\}$, which is a domain in $\C$, thus a holomorphic extension of the right hand side is also a holomorphic extension of the left hand side. Let us call it $\Psi$. Unfortunately we cannot conclude that $\Psi(s)=\int_S \E \abs{yR}^s \nu_s(dy)$ for $s \ge \beta$,  this is caused by the extra dependence on $z$ appearing in $\nu_z$ { and this is why the Landau lemma} cannot be used here.
We only know, that $\Psi(s)= \int_S \E \abs{yR}^s \nu_s(dy)$ for $s < \beta$. But that at least implies that if $\Psi$ exists on $B_\delta(\beta)$, then
$$ \lim_{s \uparrow \beta} \int_S \E \abs{yR}^s \nu_s(dy) = \lim_{s \to \beta} \Psi(s) < \infty .$$

So the following lemma gives the contradiction if $\E \abs{R}^\beta = \infty$, implying that $K>0$ at the outset.

\begin{lemma}\label{landau lemma}
 If $\E \abs{R}^\beta = \infty$, then
\begin{equation}
\label{lim:g} \lim_{s \to \sigma} \int_S \E \abs{yR}^s \nu_s(dy) = \infty.
\end{equation}
\end{lemma}

\begin{proof}
Now if $\E \abs{R}^\sigma = \infty$, then there is $x_0 \in S$ with
$\E\abs{x_0R}^\sigma=\infty$. By Lemma
\ref{infinite on a open set}, then there is already $\epsilon >0$ such that $\E
\abs{yR}^\sigma = \infty $ for all $y \in B_\epsilon(x_0) \cap S$. By Lemma
\ref{lemma:supp:nu},
\begin{align*}
\liminf_{s \to \sigma} \int_S \E \abs{yR}^s \nu_s(dy) \ge & \liminf_{s \to \sigma} \int_S \Erw{\abs{yR}^s
\1[\{\abs{yR}\ge 1 \}] } \nu_s(dy) \\
\ge & \liminf_{s \to \sigma} p \int_S \Erw{\abs{yR}^s
\1[\{\abs{yR}\ge 1 \}] } \phi(dy) \\
\ge & p \int_{B_\delta(x_0) \cap S} \liminf_{s \to \sigma} \Erw{\abs{yR}^s
\1[\{\abs{yR}\ge 1 \}] } \phi(dy) = \infty.
\end{align*}
In the penultimate line, we used Fatou's lemma.
\end{proof}

\begin{lemma}\label{infinite on a open set}
{If $ \Erw{\abs{xR}^s} = \infty$ for some $x \in S$, then there
exists $\epsilon >0$  such that $ \Erw{\abs{yR}^s} = \infty$
for all $y \in B_\epsilon(x)$. }
\end{lemma}

\begin{proof}
\textsc{Case 1}: Suppose that there is an open $U \subset \R^d$ such that $\E
\abs{yR}^s=\infty$ for all $y \in U \setminus\{0\}$. Then of course $\E
\abs{\frac{y}{\abs{y}}R}^s = \infty$ and the image of $U \setminus \{0\}$ on the
sphere $S$ is open.

\textsc{Case 2}: If not, then
$$B:=\{ y \in \R^d: \ \E \abs{yR}^s < \infty \}, $$ is dense. But then there
is already a basis $y_1, \dots, y_d$ of $\R^d$ such that $\E\abs{y_iR}^s<\infty$
for all $1 \le i \le d$, and so $\E \abs{yR}^s < \infty$ for every $y \in \R^d$, but this is contradiction.
To construct such a basis just note that the complement of the linear hull
$\mathrm{lin}(y_1, \dots, y_k)$ is a hyperplane, so we can always find a
vector in $B$ which is independent of the previously chosen basis vectors.

\end{proof}

\section{Applications of the generalized smoothing transform}\label{sect:example}

\subsection{ Multivariate examples from kinetic gas theory}\label{sect:example_mathes}
{
\subsubsection*{Particle velocity in a homogeneous gas}
Consider the distribution $V \in \R^d$ of particle velocities in a homogeneous gas. If two particles collide, the interacting particles change their speeds $v_1$ and $v_2$ to the post-collision speeds $v'_1$ and $v'_2$ according to the formula
$$ v_1'=C_1^{(1)} v_1 + C_2^{(1)} v_2, \quad v_2'=C_1^{(2)} v_2 + C_2^{(2)}v_1,$$
where $(C_1^{(1)}, C_2^{(1)})$ and $(C_1^{(2)}, C_2^{(2)})$ are $d \times d$-matrices, transmitting the interaction and depending on the scattering angle. { To obtain a stochastic fixed point equation, we interpret the interacting particles as randomly picked, so we may assume that $v_1$ and $v_2$ are iid and distributed according to $V$, and independent of the scattering angle, say the matrices, which can also be assumed to be random.}
So if $(C_1, C_2)$ denotes a generic copy of $(C_1^{(1)}, C_2^{(1)})$, in the equilibrium the speed distribution $V$ would satisfy the stochastic fixed point equation
$$ V \eqdist C_1 V_1 + C_2 V_2,$$
with $V, V_1, V_2$ iid and independent of $(C_1, C_2)$.

\subsubsection*{A particular model: Maxwell-type gas}
Bassetti and Matthes \cite[Example 6.1]{Bassetti2011} study the following model for a Maxwell-type gas:
$$ C_1 = UY^\top Y, \qquad C_2=\mathrm{Id} - UY^\top Y ,$$
where $Y$ is a random unit row vector, modelling the line of collision and $U \in \R_*^+$ is a random variable modelling the inelasticity. They assume $\phi=\mathcal{L}(Y)$ to be the uniform distribution on $S$, and independent of $U$. Moreover, they assume that
$$ \E\big[U(1-U)\big]=0 \quad \text{ and } \E\big[U^2(1-U)^2\big]< \E\big[{\abs{\skalar{w_1,Y}}^4}\big] \E{U^2},$$
where $w_1$ denotes the unit vector $(1,0, \dots,0)$.
Let's finally assume that the distribution $\psi$ of $\log U$ is nonarithmetic.
Then we have that for all $x \in S$,
$$ \P{\esl{xC_1} \in \cdot, \log \abs{xC_1} \in \cdot} = \phi \otimes \psi,$$
and since $\mathcal{L}(C)=\frac{1}{2} \mathcal{L}(C_1) + \frac{1}{2}\mathcal{L}(C_2)$,
$$ \forall_{x \in S} \P{\esl{xC} \in \cdot, \log \abs{xC} \in \cdot} \ge \frac12 \phi \otimes \psi, $$
thus \eqref{minorization} holds. Since $\phi$ is the uniform distribution $l$ on $S$, also \eqref{notvanish} holds. The same would be true for any distribution $\phi$ having a density w.r.t. to $l$.

By the very definition of $C_1, C_2$, $\Sigma = \Id$ satisfies \eqref{covariance}.
Now lets consider $m$.
We have
\begin{align*}
m(2) = & 2 \E\big[ \sup_{x \in S} \abs{Cx}^2\big] =  \E\big[{ \sup_{x \in S} (\skalar{C_1x, C_1x} + \skalar{C_2x, C_2x}) }\big] =  \E\big[{\sup_{x \in S} x^\top (C_1^\top C_1 + C_2^\top C_2)x}\big] \\
\le  & \sup_{x \in S} x^\top \Id x - 2 \E\big[U(1-U)\big]\E\big[{\inf_{x \in S} x Y^\top Y x^\top}\big] = x \Id x^\top = \skalar{x,x} = 1.
\end{align*}
Moreover, it follows from \cite[6.1.2]{Bassetti2011}, that $m(4) < 1$. Thus $\alpha =2$, and the assumptions of Proposition \ref{EU:hom} are satisfied, giving a unique solution $V$.

In \cite[6.1.4]{Bassetti2011}, it is shown that $m(s) > 1$ for some $s >4$. This gives the existence of $\beta$ with $m(\beta)=1, m'(\beta)>0$ for some $\beta >4$. Thus our Theorem \ref{main theorem2} is applicable in this situation.

\subsubsection*{Precise tail estimates for the Maxwell-type gas.}

\begin{lemma}
In the situation above, $s_\infty = \sup \{s >0 \ : \ \E U^s < \infty \}$, and if $l$ denotes the uniform distribution on the sphere,
$$ e_s \equiv 1, \quad \nu_s = l \quad \forall s < s_\infty,$$
as well as
$$ m(s) = \Erw{U^s + \abs{w_1-UY}^s}.$$
\end{lemma}

\begin{proof}
We have $$\E \norm{C}^s = \frac12 \E \norm{C_1}^s + \frac12 \E \norm{C_2}^s \le \frac12 \E U^s + \frac12\big( 1 + \E U^s\big) = \frac12 + \E U^s ,$$
and also $E \norm{C}^s \ge \frac12 \E \norm{C_1}^s = \frac12 \E U^s $.

Since $Y$ has the uniform distribution $l$ on $S$,
we infer that for all $x \in S$
$$ \mathcal{L}(x - xUY^\top Y)=\mathcal{L}(w_1 - UY), \qquad \mathcal{L}(xUY^\top Y) = \mathcal{L}(UY), $$
thus for all $s < s_\infty$, $f \in C(S)$,
\begin{align*}
T_s f(x) = & \Erw{f(\esl{xC}) \abs{xC}^s } = \frac12 \Erw{f(\esl{xC_1}) \abs{xC_1}^s } + \frac12 \Erw{f(\esl{xC_2}) \abs{xC_2}^s } \\
= &  \frac12 \Erw{f(\esl{Y}) \abs{U}^s } + \frac12 \Erw{f(\esl{w_1-UY}) \abs{w_1-UY}^s }   .
\end{align*}
Since the RHS does not depend on $x$, $e_s=\1[S]$ and $\nu_s=l$ for all $s < s_\infty$.

The formula for $m(s)$ follows from Proposition \ref{prop:properties es nus}, using that $T_s \1[S] = \kappa(s)$ and $m(s)=2 \kappa(s)$.
\end{proof}

Given the distribution of $U$, $\beta$ is explicitly computable from this formula.

Finally, in this particular example, identity \ref{identity} reduces to
$$ \int_S \E \abs{yV}^s \nu_s(dy) = \E \abs{YV}^s = (1- m(s))^{-1} \Erw{\abs{Y(C_1 V_1 +C_2 V_2)}^s - \abs{YC_1V_1 }^s - \abs{YC_2 V_2}^s}, $$
with $V_1, V_2$ iid and distributed like the solution $V$. Thus the LHS is the Mellin transform of the random variable $\abs{YV}$, and our holomorphic argument works as in the situation of similarities. {Since by \cite[6.1.4]{Bassetti2011}, there exists $s < s_\infty$ such that $\E \abs{V}^s = \infty$, we deduce that necessarily $\E \abs{V}^\beta = \infty$.}

Thus an application of Theorem \ref{main theorem2} yields that
$$ \lim_{t \to \beta} t^\beta \P{xV >t} = K > 0 \qquad \forall x \in S.$$

\subsubsection{A second example with similarities}
In the same paper by Bassetti and Matthes, there is also an example that considers the multivariate smoothing transform with similarities, \cite[Example 6.2]{Bassetti2011}:
It is stated as follows: Let $A, B$ random matrices in $SO(\R^d)$, and $a,b$ be non-negative random variables, independent of $(A,B)$ and satisfying
$$ m(2)=\E \big[ a^2 + b^2\big] =1, \qquad m(s) = \E \big[a^s +b^s\big] < 1$$
for some $2 < s < 3$. Then they consider the homogeneous smoothing transform associated with
$$ C_1 = aA, \quad C_2 = bB.$$
One readily checks that, since $A,B$ are orthogonal,
$$ \E\big[ C_1 \Id C_1^\top + C_2 \Id C_2^\top\big] = \E\big[ a^2 AA^\top + b^2 BB^\top\big] = \E \big[ {a^2 +b^2} \Id\big] = \Id.$$
Thus by Proposition \ref{EU:hom} there is a unique solution $V$, and if there is $\beta>2$ such that $\E\big[ a^\beta +b^\beta\big]=1$, then by Theorem \ref{thm: similarities}, $V$ has heavy tails with tail index $\beta$.


}

\subsection{ The inhomogeneous smoothing transform with $\beta=2$}
Consider the inhomogeneous fixed point equation
\begin{equation} \tag{SFPE} R \eqdist \sum_{k=1}^N C_k R_k + Q .\end{equation}
Let $\E Q =0$, $0 <\Var\ Q < \infty$, and assume $\beta =2$. Then by Proposition \ref{EU:inhom}, there is a unique solution $R$ with $\E R =0$. Notice that this $R \equiv 0$ in the case of the associated homogeneous equation. Also, $R \equiv Q \equiv 0$ would be the only possible solution when only considering solutions on the positive cone as in \cite{BDG2011,Mirek2011a}. But in our model, $R$ is certainly nondegenerate due to the influence of $Q$.

This is a very nice example to study, because another feature, which is only possible in the multidimensional setting, occurs: If one would allow for noninvertible matrices $C$, and
$$ [\supp \ C] \cdot (\supp \ Q) = \{0\},$$
where $[\supp\ C]$ denotes the smallest closed semigroup generated by $\supp\ C$, and $(\supp\ Q)$ the linear hull of $\supp\ Q$; then the solution to \eqref{SFPE} is given by $R\eqdist Q$.

In our model, we only consider regular matrices, so this situation will not occur. Nevertheless, one may expect that properties of $R$ strongly depend on the \emph{interaction} between the distributions of $C$ and $Q$.
In the case $\beta =2$, this can be made explicit (by the way showing, that $K$ is positive in this case):

\begin{lemma}\label{lemma:beta2}
Consider the inhomogeneous SFPE \eqref{SFPE}. Let the assumptions of Theorem \ref{main theorem} hold, and assume  $\E Q = 0$, $\Var Q >0$ and $\beta=2$.
Then the unique solution $R$ with $\E R =0$ satisfies
$$ \lim_{t \to \infty} t^2 \P{xR > t}= \frac{e(x)}{4 l_2} \int_S \E(yQ)^2 \nu_2(dy)  >0,$$
for all $x \in S$.
\end{lemma}

\begin{proof}
We calculate for $y \in S$, using the independence of $R_1, \dots, R_N$ and $(C_1, \dots, C_N,Q)$ as well as $\E R=0$ :
\begin{align*}
& \Erw{\abs{\sum_{i=1}^N yC_i R_i + yQ}^2 - \sum_{i=1}^N \abs{yC_iR_i}^2 }  \\
= & \Erw{\sum_{i \neq j}(yC_iR_i)(yC_jR_j)+\sum_{i=1}^N (yC_iR_i) yQ + (yQ)^2} \\
= & \sum_{i \neq j} \Erw{ \E\left[ \left. (yC_iR_i)(yC_jR_j) \right| C_i, C_j \right] } + \sum_{i =1}^N \Erw{ \E\left[ \left. (yC_iR_i)(yQ) \right| C_i, Q \right] } + \E (yQ)^2 \\
= & \sum_{i \neq j} \Erw{  (yC_i \E R)(yC_j \E R) } + \sum_{i =1}^N \Erw{ (yC_i \E R)(yQ)  } + \E (yQ)^2 \\
= &  \E(yQ)^2.
\end{align*}
Now by Proposition \ref{prop:implicit renewal}, the limit is given by $\frac{e(x)}{4 l_2} \int_S \E(yQ)^2 \nu_2(dy)$. It is shown in Lemma \ref{lemma:supp:nu}, that $\supp(\nu_2)=S$, thus the last integral is indeed positive.
\end{proof}

To interpret this formula, remember that $\nu_2$ is the invariant measure of the operator $T_2$ defined by
$$ T_2 f(x) = \Erw{\abs{xC}^2 f(\esl{xC})}.$$
So it is directly connected to the action of $C$.


\subsection{On contraction conditions}

In their survey article \cite{NR2006}, Neininger and R\"uschendorf introduce several probability metrics, namely Zolotarev and minimal $l_p$-metric, and discuss the resulting conditions for $S$ to be a contraction on $\M_2(0)$.

They list several conditions for $S$ to have a unique fixed point with zero expectation and finite variance, namely
\begin{enumerate}
  \item $\E \sum_{k=1}^N \norm{C_k}^2 <1$, which is the one we used and comes from the Zolotarev metric,
  \item $ \sum_{k=1}^N \E \norm{C_k^\top C_k} < 1$, which comes from the minimal $l_2$-metric, and
  \item $ \norm{\sum_{k=1}^N \E C_k^\top C} < 1$.
\end{enumerate}
The last one is the weakest. In the context of the contraction method (see there for details), they pose the question, which is the best, i.e. weakest condition possible. See \cite[Problem 3.2]{NR2006}.

We cannot resolve their question, but our example above gives a good hint, which may be the optimal condition: We suppose it is $m(2)<1$; for the example above shows, that if $m(2)=1$, the unique solution with zero expectation has infinite variance. In particular, if $m(2)=1$, there is no solution with finite variance.

\appendix

\section{Reduction to identically distributed weights}\label{sect:proofs}

\begin{proposition}\label{prop:permutation}
Let $\sigma$ be a r.v. with uniform distribution on the symmetric group of order
$N$, independent of all other occuring r.v.s. Then $R$ (with iid copies $R_1,
\dots, R_N$) is a solution to \eqref{SFPE} if and only if
\begin{equation}\label{SFPE:permutation}
R \eqdist \sum_{i=1}^N C_{\sigma(i)} R_i + Q ,
\end{equation}
i.e. $R$ solves the SFPE associated with the vector $(C_{\sigma(1)}, \dots,
C_{\sigma(N)},Q)$.

Moreover, $r \in \R^d$ satisfies \eqref{ev} for $(C_1, \dots, C_N,Q)$ if and
only if $r$ satisfies \eqref{ev} for $(C_{\sigma(1)}, \dots, C_{\sigma(N)},Q)$.
\end{proposition}

\begin{proof}
Let $R_1, \dots, R_N$ be iid, independent of $(C_1, \dots,C_N,Q)$ and
$\sigma$ as defined above. Then for all $f \in
C_b(\R^d)$,
\begin{eqnarray*}
 \E f\left( \sum_{k=1}^N C_k R_k + Q \right)
&= & \E f \left( \sum_{k=1}^N C_{\sigma(k)} R_{\sigma(k)} + Q \right) \\
&= &\E \left[  \E \left( \left. f \left( \sum_{k=1}^N C_{\sigma(k)} R_{\sigma(k)}
+ Q \right) \right| (C_1,\dots,C_N,Q),\sigma \right] \right) \\
&= &\E \left[  \E \left( \left. f \left( \sum_{k=1}^N C_{\sigma(k)} R_{k}
+ Q \right) \right| (C_1,\dots,C_N,Q),\sigma \right] \right) \\
&= & \E \left[  f \left( \sum_{k=1}^N C_{\sigma(k)} R_{k}
+ Q \right) \right].
\end{eqnarray*}
Now the distribution of $R$ is a solution to the original equation, if $\E
f(R) $ equals the first line, and it is a solution to the \emph{permuted}
equation, if $\E
f(R) $ equals the last line; so both equations are indeed equivalent.

The statement about the eigenvalues can be readily checked by taking
expectations in both equations.
\end{proof}

\begin{remark}Proposition \ref{prop:permutation} shows that we may
w.l.o.g. assume that $C_1, \dots, C_N$ are identically distributed (but dependent) with generic
copy $C$.
\end{remark}

\section{Existence and Uniqueness of Solutions}\label{sect:EU}

The aim of this section is to show existence and uniqueness of solutions to \eqref{SFPE} and \eqref{SFPE'}, given by Propositions
\ref{EU:hom} and \ref{EU:inhom}.  For this purpose we consider the mapping $\S$ as a contraction operator on an appropriate complete metric space consisting of probability measures. Then both existence and uniqueness of solutions follow from the Banach fixed point theorem.

\subsection{The Zolotarev metric}
The right metric on measures for our purpose is the Zolotarev metric $\zeta_s$, introduced by Zolotarev \cite{Zolotarev1976}, and { considered} in this context by Rachev and R\"uschendorff \cite{RR1995}.

For $X,Y$ two random variables on a common probability space, we define
\begin{equation} \label{Zolotarev metric1}
\zeta_s (Z,Y) := \sup \{ \abs{ \Erw{f(X)- f(Y)}} \ : \ f \in \mathfrak{D}_s\},
\end{equation}
where $\mathfrak{D}_s$ is the space of functions defined as follows
\begin{equation}
\label{Zolotarev Ds2} \mathfrak{D}_s = \{ f \in C^k(\R^d) \ : \ \forall_{x,y \in \R^d}\ \abs{f^{(k)}(x) - f^{(k)}(y)} \le C \abs{x-y}^{s-k}\} \qquad \text{ with } k=\lceil s-1 \rceil,
\end{equation}
and
  $\lceil \alpha \rceil$ denotes the smallest integer $\ge\alpha$.
  Observe that for $0 < s \le 1$, $\mathfrak{D}_s$ is the set of $s$-H\"older functions on $\R^d$.

 For properties of the Zolotarev
metric $\zeta_s$ see \cite[Section 2]{NR2004}; or for a recent treatment of the Zolotarev metric on general separable Hilbert spaces see \cite{Drmota2008}.  It turns out that on an appropriately defined measure space the Zolotarev metric is complete and the mapping $\S$ is a contraction. The details are as follows.

{ Let ${\bf{a}}=(a_1, \dots, a_d) \in \N_0^d$ be a $d$-dimensional multi-index, and set $\abs{\bf{a}}=a_1 + \dots + a_d$. For a probability measure $\nu$ on $\R^d$, denote by $M_k(\nu)$ the sequence of mixed integer moments up to order $k$, i.e.
$$ M_k(\nu) = \left( \int_{\R^d} x_1^{a_1} \cdots x_d^{a_d} \, d\nu(x_1, \dots, x_d) \right)_{\abs{\bf{a}} \le k} . $$}

 For a fixed mixed moment sequence $M_k$, we introduce the following subsets of the space of probability measures on $\R^d$ (as before, $k=\lceil s-1 \rceil$):
\begin{equation}
\mathfrak{M}^d_s(M_k) \ := \left\{ \nu \ : \ \int_{\R^d} \abs{x}^s \nu(dx) < \infty, \ M_k(\nu)=M_k \right\} .
\end{equation}
Thus, we consider  subspaces of measures such that we fix integer moments, i.e. the expectation if $1 < s \le 2$, and also the covariance matrix if $2 < s \le 3$, and so on.

Then we have the following Lemma:
\begin{lem}\label{complete metric space proposition}
The metric spaces $(\M_s^d(M_k), \zeta_s)$ are complete for $s>0$. Moreover,  if $m(s)<1$, then for some $n_0$ the operator $\S^{n_0}$ is a contraction on the space $(\M_s^d(M_k), \zeta_s)$.
\end{lem}
\begin{proof} The first part of the Lemma was indeed proved in \cite[Theorem 5.1]{Drmota2008}. However for  reader's convenience we give some ideas of the proof.

Let $\nu_n$ be a Cauchy sequence in  $\M^d_s(M_k)$ with respect to the Zolotarev metric. By \cite[Theorem
5]{Zolotarev1976}, it is also a Cauchy sequence in the Prokhorov metric, thus
weakly convergent towards a probability measure $\nu$ - observe that characteristic functions belong to $\mathfrak{D}_s$ for any $s>0$. By \cite[Theorem
6]{Zolotarev1976}, the sequence of $\int \norm{x}^s \nu_n(dx)$ is also a (real)
Cauchy sequence, thus bounded, implying $\int \norm{x}^s \nu(dx) < \infty$.
For $s >1$, this also yields uniform integrability of the moments of order up to s, thus $\nu$ again is in
$\M^d_s(M_k)$.
It remains to prove that $\zeta_s(\nu_n, \nu) \to 0$, for which we refer to \cite[Lemma 5.6]{Drmota2008}.

\medskip

The second part follows directly from  the inequality proved in \cite[Lemma
3.1]{NR2004}:
$$
 \zeta_s(\S\nu, \S\eta) \le \left( \E
\sum_{i=1}^N \norm{C_i}^s \right) \zeta_s(\nu, \eta) = N \E \norm{C}^s
\zeta_s(\nu, \eta).$$
\end{proof}


\subsection{Proof of Existence and Uniqueness}
Now we are ready to prove both Propositions. 

\begin{proof}[Proof of Propositions \ref{EU:hom} and \ref{EU:inhom}]

If $m(s)<1$ for $s\le 2$, the mapping $\S^{n_0}$ is a contraction on { $\M_s^d$, $\M_s^d(r)$ resp. $\M_s^d(r, \Sigma)$ }which is a complete metric space (Lemma \ref{complete metric space proposition}). Therefore by the Banach fixed point theorem  there exists a unique fixed point $R$ of $\S$, being a unique solution of \eqref{SFPE}, resp. \eqref{SFPE'}.
Notice that for { $1 < s \le 2$ $(2 < s \le 3)$}, we have to fix the expectation {(and the covariance matrix)} in order for $\S$ to be a
contraction. Moreover, to ensure that $\S$ maps $\M_s^d(r)$ into $\M_s^d(r)$ { ($\M_s^d(r, \Sigma)$ into $\M_s^d(r, \Sigma)$),
we have to know that $r$ satisfies  condition \eqref{ev} ($\Sigma$ satisfies \eqref{covariance} for the homogeneous equation).}
Observe further that in the homogeneous case $R=0$ is a trivial solution of \eqref{SFPE'}. Since for $s\le 1$, $\d_0$ is an element of    $\M_s^d$, there are no other solutions of this equation. Thus \eqref{SFPE'} possesses nontrivial solutions with finite expectation only if $\a\ge 1$ and thus $s>1$.

\medskip

To prove the second part of both Propositions, observe  we already know that
all moments of $R$ up to some $s_0 >
\alpha$ are finite. By convexity of $m$, $m(s)<1$ for all $s \in
(s_0,\min\{\beta, s_\infty\})$. Then arguing as above, Lemma
\ref{complete metric space proposition} assures the existence of a solution with
finite $s$-moments, which by uniqueness coincides with $R$.

If $s >2$, then $\S$ is a contraction only on subspaces of measures with
fixed mixed integer moments of higher order; and we have to check whether $\S$ is still a self-mapping of { these} subspaces. But this is no problem, since Lemma \ref{complete metric space proposition} applied to $p=\lfloor s \rfloor$, (here and subsequently $\lfloor \alpha \rfloor$ denotes the biggest integer $\le\alpha$) gives finiteness of all mixed moments of $R$ up to order $p$, and since $R$ is a solution, then also $\S(M_p(R))=M_p(R)$. Starting with $p=2$, we can proceed in the same manner at every integer step. 
\end{proof}
%
%


\section{Proofs of Lemmas \ref{ew6.1} and \ref{lemma:dri}}
\label{app:lemmas}

\begin{proof}[Proof of Lemma \ref{ew6.1}]
We estimate separately
\begin{equation}
\label{e1}
\big| f(x_1+\ldots + x_n + q) - f(x_1+\ldots+ x_n)
\big| \le c |q|^{\eps}
\end{equation}
and
\begin{equation}
\label{e2}
\bigg| f(x_1+\ldots + x_n ) - \sum_{j=1}^n f(x_j)
\bigg| \le c \sum_{i\not=j}|x_i|^{\eps}|x_j|^{\eps}.
\end{equation}
On one hand \eqref{e1} is bounded and on the other it is dominated by $c|q|$. Hence
$$
\big| f(x_1+\ldots + x_n + q) - f(x_1+\ldots+ x_n)
\big| \le c(f) |q|^{\eps}.
$$
For \eqref{e2}, we first prove
\begin{equation}
\label{e3}
|f(x_1+x_2) -  f(x_1) - f(x_2) | \le c(f)|x_1|^{\eps}|x_2|^{\eps}
\end{equation}
and then we proceed by induction as follows
\begin{multline*}
\bigg| f\big( x_1+\ldots + x_n \big) - \sum_{j=1}^n f(x_j) \bigg|\\ \le
\big|  f\big( x_1+\ldots + x_n \big) - f(x_1) - f\big( x_2+\ldots + x_n  \big) \big|
+ \bigg| f\big( x_2+\ldots + x_n  \big) - \sum_{j=2}^n f(x_j) \bigg|\\
\le c|x_1|^{\eps}|x_2+\ldots+ x_n|^{\eps} + { c \sum_{i\not= j; i,j \ge 2}} |x_i|^{\eps}|x_j|^{\eps} \le  c \sum_{i\not= j} |x_i|^{\eps}|x_j|^{\eps}.
\end{multline*}
For \eqref{e3} we write $u=(u_1,\ldots,u_d), y=(y_1,\ldots,y_d)\in \R^d$ and
\begin{eqnarray*}
f(u+y) - f(u) &=& \int_0^1 \frac d{ds} f(u+sy)ds =  \int_0^1 \sum_{j=1}^d \partial_j  f(u+sy)y_j ds,\\
f(y) &=& f(y) - f(0) = \int_0^1 \sum_{j=1}^d \partial_j  f(sy)y_j ds.
\end{eqnarray*}
Now notice that
$$
\partial_j f(u+sy) - \partial_j f(sy) = \int_0^1 \sum_{k=1}^d\partial_k \partial_j  f(ru+sy)y_j u_k dr.
$$
Hence
$$
|f(u+y) - f(u) - f(y)| \le \int_0^1\int_0^1 \sum_{j,k=1}^d \big|\partial_k\partial_j  f(ru+sy)\big| |u_k||y_j| ds dr \le C |u||y|.
$$
On the other hand $|f(u+y) - f(u) - f(y)|$ is bounded, so \eqref{e3} follows.
\end{proof}

{  In several places, we will use the inequality
\begin{equation} \label{ineq1}
\left( \sum_{i=1}^n y_i \right)^s \le n^{0 \vee (s-1)} \sum_{i=1}^n y_i^s,\end{equation}
valid for $y_i \ge 0$ and $s \ge 0$. For $s \le 1$, this follows from the subadditivity of $x \mapsto x^s$. 
For $s>1$, it follows from Jensen inequality, when applied to the discrete probability measure $\frac{1}{n} \sum_{i=1}^n \delta_{y_i}$ and the convex function $x \mapsto x^s$.

}

\begin{proof}[Proof of Lemma \ref{lemma:dri}]
Writing $g=tk$, $t\in \R^+$, $k\in K$ we have to show that
$$
\sum_{n=-\8}^{\8} \sup_{e^n\le t\le e^{n+1}; k\in K}  |\psi_f(tk)| <\8.
$$
Assuming ${\rm supp}f\subset \R^d\setminus B_{\eta}(0)$ and $e^n< t\le e^{n+1}$
 by Lemma \ref{ew6.1}
 we have
\begin{eqnarray*}
|\psi_f(tk)|&\le& t^{-\b} \E\bigg| f\bigg(tk\bigg( \sum_{i=1}^N C_i R_i + Q\bigg)\bigg)
-\sum_{i=1}^N f(tk C_i R_i)\bigg|\\
&\le& c t^{-\b}\E\bigg[\bigg(
\sum_{i\not= j} |tkC_iR_i|^{\eps} |tk C_j R_j|^{\eps} + |tkQ|^{\eps}\bigg) \cdot {\bf 1}_{\{\sum|C_iR_i| +|Q| > \eta t^{-1}\}}
\bigg]\\
&\le& c e^{-(\b-2\eps)n}\sum_{i\not= j }\E \big[ |C_i R_i|^{\eps}|C_j R_j|^{\eps}{\bf 1}_{\{\sum|C_iR_i| +|Q| > \eta e^{-n-1}\}}\big]\\
&&+ c e^{-(\b-\eps)n}\E\big[ |Q|^{\eps} {\bf 1}_{\{\sum|C_iR_i| +|Q| > \eta e^{-n-1}\}}\big].
\end{eqnarray*}
Denote the first expression above by $I(n)$ and the second one by $II(n)$.
Let the random variable $n_0$ be defined as $n_0 = \lceil \log\eta - 1 -\log(\sum_j|C_iR_i|+|Q|)\rceil$. { The value of a constant $c_{\bullet}$ may change from line to line, but is finite and depends only on the indicated variables.} Then
{ \begin{eqnarray*}
\sum_n II(n) &=& c \sum_n e^{-(\b-\eps)n}\E |Q|^{\eps} {\bf 1}_{\{\sum|C_iR_i| +|Q| > \eta e^{-n-1}\}}\\
 &\le& c\E \bigg[|Q|^{\eps} \cdot \sum_{n\ge n_0}e^{-(\b-\eps)n} \bigg]
 \le c \E \bigg[|Q|^{\eps} e^{-(\b-\eps)n_0} \bigg]\\
 &\le& c_\eta \E \bigg[ \bigg( \sum_{i=1}^N |C_iR_i| + |Q|
 \bigg)^{\b-\eps} |Q|^{\eps}  \bigg]\\
 &\le& c_\eta N^{0 \vee (\beta - \epsilon - 1)}\bigg( \sum_{i=1}^N \E \big[ \|C_i\|^{\b-\eps}|R_i|^{\b-\eps}|Q|^\eps \big] + \E|Q|^\b \bigg),
\end{eqnarray*} where the last line follows by an application of \eqref{ineq1}. }The last expression is finite since
$$
\E \big[ \|C_i\|^{\b-\eps}|R_i|^{\b-\eps}|Q|^\eps \big]
= \E \big[|R_i|^{\b-\eps}\big]\E\big[ \|C_i\|^{\b-\eps}|Q|^\eps \big]
\le \E \big[|R_i|^{\b-\eps}\big]  \Big(\E \|C_i\|^{\b}\Big)^{\frac {\b-\eps}\b}
\Big(\E |Q|^{\b}\Big)^{\frac {\eps}\b}.
$$
Next, recalling that  $i\not=j$ and $C_i$ are identically distributed  we write
\begin{eqnarray*}
\sum_n I(n)&\le& cN^2\sum_n e^{-(\b-2\eps)n} \E \big[ |C_1 R_1|^{\eps}|C_2 R_2|^{\eps}{\bf 1}_{\{\sum|C_1R_1| +|Q| > \eta e^{-n-1}\}}\big] \\
&\le& c_N\E\bigg[ \sum_{n\ge n_0} e^{-(\b-2\eps)n}   |C_1 R_1|^{\eps}|C_2 R_2|^{\eps} \bigg]\\
&\le& c_N \E\Big[ e^{-(\b-2\eps)n_0}   |C_1 R_1|^{\eps}|C_2 R_2|^{\eps} \Big]\\
&\le& c_{N,\eta} \E \bigg[\bigg(\sum_{k=1}^N |C_k R_k| + |Q|\bigg)^{\b-2\eps} |C_1 R_1|^{\eps} |C_2 R_2|^{\eps}
\bigg]\\
&\le& { c_{N,\eta} N^{0 \vee (\beta - 2 \epsilon -1)}  \E \bigg[\bigg(\sum_{k=1}^N |C_k R_k|^{\b-2\eps} + |Q|^{\b-2\eps}\bigg) |C_1 R_1|^{\eps} |C_2 R_2|^{\eps}
\bigg]}\\ 
&\le& { c_{N,\eta,\b} \E \bigg[\sum_{k=1}^N |C_k R_k|^{\b-2\eps}  |C_1 R_1|^{\eps} |C_2 R_2|^{\eps}\bigg] + \E \bigg[ |Q|^{\b-2\eps} |C_1 R_1|^{\eps} |C_2 R_2|^{\eps}  \bigg] }\\ 
&\le& { c_{N,\eta,\b}\Big[
 \Big( \E |R|^{\eps} \Big)^2 \Big( \E |Q|^\b\Big)^{\frac{\b-2\eps}{\b}} \Big( \E \|C_1\|^\b\Big)^{\frac{2\eps}{\b}} } \\
 & & {
+  \sum_{k=1}^N \Big(\E\|C_k\|^\b\Big)^{\frac{\b-2\eps}{\epsilon}} \Big( \E \|C_1\|^\b\Big)^{\frac{2\eps}{\b}}
 \cdot \E\big[ |R_k|^{\b-2\eps} |R_1|^{\eps}|R_2|^{\eps}\big]\Big],}
\end{eqnarray*}
{ where we used again \eqref{ineq1}, and the general H\"older inequality in the last line.}
Only finiteness of the last term requires some proof. If $k\notin\{1,2\}$, then
$$
\E\big[ |R_k|^{\b-2\eps}|R_1|^{\eps}|R_2|^{\eps} \big] =
\E |R_k|^{\b-2\eps}\E |R_1|^{\eps}\E |R_2|^{\eps},
$$
otherwise
$$\E\big[ |R_k|^{\b-2\eps}|R_1|^{\eps}|R_2|^{\eps} \big] =
\E |R_1|^{\b-\eps}\E |R_2|^{\eps}, $$
and both expressions are finite.
\end{proof}

\section{Proofs of Propositions \ref{prop:Goldie lemma} and \ref{extension of moments}}
\label{app:proofs}

To prove Proposition \ref{prop:Goldie lemma} we need the following lemma:


\begin{lemma}\label{lemma2}
For all $0<s<s_{\8}$, $$s\int_0^\infty t^{s-1} \abs{ \P{\max_i \abs{yC_iR_i} >t}  -
N\P{\abs{yCR} >t} } dt \le \eqref{E2}.$$ If \eqref{E2} is finite, then
\begin{equation}
z\int_0^\infty t^{z-1} \left( \P{\max_i \abs{yC_iR_i} >t}  - N\P{\abs{yCR} >t}
\right) dt = \Erw{\abs{\max_i yC_iR_i}^z- \sum_{i=1}^N
\abs{yC_iR_i}^z },
\end{equation}
for all $z\in\mathbb{C}$ such that $0 < \Re z < s$.
\end{lemma}

\begin{proof}
If \eqref{E2} is finite, then the function $$zt^{z-1}\left( \sum_{i=1}^N
\1[\{\abs{yC_i R_i >t} \}]- \1[\{\max_i \abs{yC_iR_i}>t \}] \right), $$ is also
absolutely integrable with respect to $\lambda \otimes \Prob$:
\begin{align*}
& \int_0^\infty \Erw{\abs{zt^{z-1}\left( \sum_{i=1}^N \1[\{\abs{yC_i R_i} >t
\}]- \1[\{\max_i \abs{yC_iR_i}>t \}] \right)}} dt \\
= & \abs{z} \int_0^\infty \Erw{t^{s-1}\left( \sum_{i=1}^N \1[\{\abs{yC_i R_i}
>t \}]- \1[\{\max_i \abs{yC_iR_i}>t \}] \right)} dt = \abs{z} \cdot\eqref{E2},\\
\end{align*}

Thus Fubini theorem applies and we obtain
\begin{align*}
& z\int_0^\infty t^{z-1} \left( \P{\max_i \abs{yC_iR_i} >t}  - N\P{\abs{yCR} >t}
\right) dt  \\
 = & \int_{0}^{\infty} z t^{z-1} \left(
 \P{\max_i \abs{yC_iR_i} > t} - \sum_{i=1}^N\P{\abs{yC_iR_i} >t}
 \right) dt \\
= & \int_{0}^{\infty} z t^{z-1} \Erw{ \1[\{\max_i \abs{yC_iR_i} >
t\}] - \sum_{i=1}^N  \1[\{\abs{yC_iR_i} >t\}] } dt \\
= &  \Erw{ \int_0^{\max_i \abs{yC_iR_i}} zt^{z-1} dt -
\sum_{i=1}^n \int_0^{\abs{ yC_iR_i}} zt^{z-1} dt } \\
= & \Erw{\abs{\max_i yC_iR_i}^z- \sum_{i=1}^N
\abs{yC_iR_i}^z }.
\end{align*}
\end{proof}

\begin{proof}[Proof of Proposition \ref{prop:Goldie lemma}]

We have
\begin{align*}
&  \int_{0}^{\infty} \abs{zt^{z-1}} \abs{
\P{\abs{\sum_{i=1}^N C_i R_i} > t} - N\P{\abs{yCR} >t} } dt \\
\le & s \int_{0}^{\infty} t^{s-1} \abs{
\P{\abs{\sum_{i=1}^N C_i R_i} > t} - \P{\abs{\sum_{i=1}^N C_i R_i} >t} } dt \\
& + s\int_0^\infty t^{s-1} \abs{
\P{\abs{\sum_{i=1}^N C_i R_i} > t} - \P{\max_i \abs{yC_iR_i} >t} } dt \\ & +
s\int_0^\infty t^{s-1} \abs{ \P{\max_i \abs{yC_iR_i} >t}  - N\P{\abs{yCR} >t} } dt
\end{align*}
Now, in view of \cite[Lemma 9.4]{Goldie1991} the first expression is bounded by \eqref{E0} and by Lemma \ref{lemma2} the second and the third one by \eqref{E1} and \eqref{E2}, respectively.

The same calculation without absolute value signs yields the identity.
Turning to the last assertion, since $\abs{zt^{z-1}}=\abs{z}t^{s-1}$,  for
all $z$ with $\Re z < s$ $$ \int_{0}^{\infty} \abs{zt^{z-1}} \abs{
\P{\abs{yR} > t} - N\P{\abs{yCR} >t} } dt < D, $$
for some finite constant $D>0$.

Now if $\gamma$ is some closed curve in the domain $0 < \Re z < s$, then
we may change the order of integrals and obtain:
\begin{align*}
& \int_\gamma \int_{0}^{\infty} {zt^{z-1}} \left( \P{\abs{yR} > t} -
N\P{\abs{yCR} >t} \right) dt dz
\\ = & \int_0^\infty \left( \int_\gamma {zt^{z-1}} dz\right) \left( \P{\abs{yR}
> t} - N\P{\abs{yCR} >t} \right) dt = 0,
\end{align*}
and holomorphicity follows.
\end{proof}

Proposition \ref{extension of moments} follow from the following lemmas

\begin{lemma}\label{lemma:E0}
For $s < s_\infty$, \ref{E0} is finite if:  $\E \abs{R}^{0 \vee
(s-1)}< \infty$.
\end{lemma}

\begin{lemma}\label{lemma:E2}
For $s < s_\infty$, \ref{E2} is finite if: There is $s/2 < \gamma < s$ s.t.
$\E\abs{R}^\gamma < \infty$.
\end{lemma}

\begin{lemma}\label{lemma:E1}
For $s < s_\infty$,  \ref{E1} is finite if:
\begin{itemize}
  \item[] \textsc{Case $s \le 2$}: \ $\E \abs{R}^\frac{s}2 < \infty$.
  \item[] \textsc{Case $s >2$}: \ $\E \abs{R}^\gamma < \infty$ for $\gamma:=
  \max\{s-1, s- \frac{s}{\lceil s/2 \rceil}\} = s-1$.
\end{itemize}
\end{lemma}

Unfortunately, proofs of all lemmas are long, therefore we first present how they imply the Proposition and later we give their  proofs.

\begin{proof}[Proof of Proposition \ref{extension of moments}]
The choice of $\epsilon$ is described by the previous lemmata. Of course $\rho
+ \epsilon < s_\infty$, the other restrictions can be reduced to: \\
\begin{enumerate}
  \item $\rho - \epsilon > \rho + \epsilon -1$ \ (Lemma \ref{lemma:E0}, Lemma
  \ref{lemma:E1} for $\rho+\epsilon >2$)
  \item $\rho - \epsilon > \frac12(\rho + \epsilon)$ \ (Lemma \ref{lemma:E2},
  Lemma \ref{lemma:E1} for $\rho+\epsilon \le 2$)
\end{enumerate}
These restrictions can always be fulfilled.
\end{proof}

\begin{proof}[Proof of Lemma \ref{lemma:E0}]
If $s \le 1$, by the inequality $\abs{a^s-b^s} \le \abs{a-b}^s$,
valid for $a,b \ge 0$,
$$ \E \abs{ \abs{\sum_{i=1}^N y C_i R_i + { yQ}}^s - \abs{\sum_{i=1}^N y C_i R_i}^s
} \le \E \abs{ Q }^s  < \infty .$$
If $s >1$, by the inequality $\abs{a^s -b^s} \le s \abs{a-b} \max
\{a^{s-1}, b^{s-1} \}$, valid for $a,b \ge 0$,
\begin{eqnarray*}
 \E \abs{ \abs{\sum_{i=1}^N y C_i R_i + { yQ}}^s - \abs{\sum_{i=1}^N y C_i R_i}^s
}
&\le & s \Erw{ \abs{Q} \max\left\{ \abs{\sum_{i=1}^N y C_i R_i + { yQ}}^{s-1},
\abs{\sum_{i=1}^N y C_i R_i}^{s-1} \right\}} \\
&\le & s \Erw{ \abs{Q} \left(  \sum_{i=1}^N \abs{C_i R_i} + \abs{Q}
\right)^{s-1}}
\\
&{ \stackrel{*}{\le}} & s (N+1)^{0 \vee (s-2)} \Erw{ \sum_{i=1}^N \abs{C_iR_i}^{s-1}\abs{Q} +
\abs{Q}^s }
\\
&\le & s (N+1)^{0 \vee (s-2)} \E \abs{R}^{s-1} \Erw{ \sum_{i=1}^N \norm{C_i}^{s-1}\abs{Q}
} + \E \abs{Q}^s \\
&\le &  s(N+1)^{0 \vee (s-2)} \E \abs{R}^{s-1} N (\E \norm{C_i}^s)^\frac{s-1}{s}
(\E\abs{Q}^s)^\frac{1}{s} + \E \abs{Q}^s.
\end{eqnarray*}
{ In *, we used subadditivity of $x \mapsto x^{s-1}$ for $s \le 2$, and \eqref{ineq1} for $s >2$.}
\end{proof}

\begin{proof}[Proof of Lemma \ref{lemma:E2}] Let $\mc{F}:=\sigma(C_1, \dots, C_N)$.
Observe that the function is nonnegative, so we may change integrals:
\begin{align}
\eqref{E2} = & \Erw{ \sum_{i=1}^N \abs{yC_i R_i}^s - \left( \max_{1 \le i \le N}
\abs{yC_i R_i} \right)^s} \nonumber \\
= &  s \int_0^\infty \left( \sum_{i=1}^N \P{\abs{yC_iR_i} > t} - \P{\max_{1 \le
i \le N} \abs{yC_i R_i} >t}  \right)  t^{s-1} dt \nonumber  \\
= & s  \Erw{ \int_0^\infty \left(\sum_{i=1}^N \P{\abs{yC_iR_i} > t | \mc{F}} -
\P{\max_{1 \le i \le N} \abs{yC_i R_i} >t | \mc{F}}\right) t^{s-1} dt }
\nonumber \\
= &s \Erw{ \int_0^\infty \left(\P{\max_{1 \le i \le N} \abs{yC_i R_i} \le t |
\mc{F}} -1 + \sum_{i=1}^N \P{\abs{yC_iR_i} > t | \mc{F}}\right)  t^{s-1} dt }.  \label{I1}
\end{align}
 Conditioned on $\mc{F}$, the r.v.s
$\abs{yC_iR_i}$ are independent, so
\begin{align*}
\P{\max_{1 \le i \le N} \abs{yC_i R_i} \le t | \mc{F}} = \prod_{i=1}^N \left( 1
- \P{\abs{yC_iR_i}>t | \mc{F}} \right)
\end{align*}
Exactly as in \cite[proof of Lemma 4.7]{JO2010a}, we use the inequality $1-x
\le e^{-x}$, valid for $x \ge 0$, to obtain
$$ \prod_{i=1}^N (1 - \P{\abs{yC_iR_i}>t | \mc{F}}) \le e^{- \sum_{i=1}^N
\P{\abs{yC_iR_i}>t | \mc{F}}} .$$

For $s/2 < \gamma < s$, such that $\E \abs{R}^\gamma < \infty,$ the Markov
inequality yields
$$ \sum_{i=1}^N \P{\abs{yC_iR_i}>t | \mc{F}} \le \sum_{i=1}^N \Erw{ \abs{yC_i
R_i}^\gamma | \mc{F}}t^{-\gamma}  \le t^{-\gamma} \E \abs{R}^\gamma
\sum_{i=1}^N\norm{C_i}^\gamma.$$

Define the function $g(x):=e^{-x} - 1 +x$ and note that $g(x)$ is increasing for $x \ge 0$.
Thus $$g(\sum_{i=1}^N \P{\abs{yC_iR_i}>t | \mc{F}} )  \le g(t^{-\gamma} \E \abs{R}^\gamma
\sum_{i=1}^N\norm{C_i}^\gamma).$$ With the change of variables
$u=t^{-\gamma} \E \abs{R}^\gamma \sum_{i=1}^N \norm{C_i}^\gamma$, the inner
integral in \eqref{I1} can be estimated by
$$ \left( \E \abs{R}^\gamma N \E\norm{C_i}^\gamma \right)^{s/\gamma}
\int_0^\infty (e^{-u}-1 +u)u^{-s/\gamma -1} du. $$

The integral is finite if $1 < s/\gamma < 2$, which was exactly our choice.
Finally, evaluating the expectation in \eqref{I1}, we see that for some finite constant
$c>0$
\begin{align*}
\eqref{E2} \le c N^{s/\gamma}\left( \E \abs{R}^{\gamma} \right)^{s/\gamma} \Erw{
\norm{C_i}^\gamma }^{s/\gamma} \le c \left( \E \abs{R}^{\gamma}
\right)^{s/\gamma} N^{s/\gamma} \Erw{\norm{C}^s}.
\end{align*}

\end{proof}

\begin{proof}[Proof of Lemma \ref{lemma:E1}] We will only consider $H(s)$, defined below
since $\eqref{E1} \le H(s) + \eqref{E2}$.
\begin{align*}
H(s): = & \E \abs{\left( \left( \sum_{i=1}^N yC_i R_i \right)^2 \right)^\frac{s}2
- \sum_{i=1}^N \abs{yC_iR_i}^s} \\
\le & \E \abs{ \left( \sum_{i=1}^N \abs{yC_iR_i}^2 + \sum_{k \neq l} yC_k R_k
yC_l R_l \right)^\frac{s}2 - \left( \sum_{i=1}^N \abs{yC_iR_i}^2
\right)^\frac{s}2} \\ & + \E \abs{\left( \sum_{i=1}^N \abs{yC_iR_i}^2
\right)^\frac{s}2 - \sum_{i=1}^N \left( \abs{yC_iR_i}^2\right)^\frac{s}2} =: H_1
+ H_2.
\end{align*}

\textsc{Case} $s > 2$:  We have
\begin{align*}
H_1 \le & \frac{s}2 \Erw{ \sum_{k \neq l} \abs{yC_k R_k}\abs{yC_l R_l} \left(
\sum_{i=1}^N \abs{yC_iR_i}^2 + \sum_{k \neq l} { \abs{yC_k R_k yC_l R_l}}
\right)^{\frac{s}2 -1} } \\
\le & s 2^{\frac{s}2 -2} \Erw{ \left( \sum_{i=1}^N \abs{yC_iR_i}^2
\right)^{\frac{s}2 -1} \left( \sum_{k \neq l} \abs{yC_k R_k}\abs{yC_l R_l}
\right) + \left( \sum_{k \neq l} \abs{yC_k R_k}\abs{yC_l R_l} \right)^\frac{s}2}
\\
\le & s 2^{\frac{s}2 -2}  N^{\frac{s}2 -2}\Erw{ \left( \sum_{i=1}^N
\abs{yC_i R_i}^{s-2} \right)\left( \sum_{k \neq l} \abs{yC_k
R_k}\abs{yC_l R_l} \right) } \\
& + s 2^{\frac{s}2 -2} (N^2-N)^{\frac{s}2 -1 } \Erw{ \sum_{k \neq l}
\abs{yC_k R_k}^\frac{s}2 \abs{yC_l R_l}^\frac{s}2 }.
\end{align*}
{ In the first line, we used the mean value theorem, in the second and third line \eqref{ineq1}.} Now by application of the (generalised)
H\"older inequality, this expression is bounded by terms of $\E\abs{R}^{s-1}$,
$\E \abs{R}$, $\E \norm{C}^s$.

To estimate $H_2$ we adapt arguments from \cite[Lemma 4.1]{JO2010a}).
Set $q =s/2$,  $p= \lceil q \rceil$ and $\gamma = q/p \in (p/p+1, 1]$.
Define the set of multiindices
$$ A_p(N):= \{ (j_1, \dots, j_N) \in \N^N \ : \ j_1 + \dots + j_N = p ; 0 \le
j_i \le p-1 \} .$$
For any sequence $y_i \ge 0$,
\begin{align*}
\left( \sum_{i=1}^N y_i \right)^q = & \left( \sum_{i=1}^N y_i \right)^{p\gamma}
\\
= & \left( \sum_{i=1}^N y_i^p + \sum_{A_p(N)} {p\choose{j_1, \dots, j_N}}
y_1^{j_1}\cdots y_N^{j_N}  \right)^\gamma \\
\le & \sum_{i=1}^N y_i^{p\gamma} + \sum_{A_p(N)} {p \choose j_1, \dots,
j_N} y_1^{\gamma j_1}\cdots y_N^{\gamma j_N}  .
\end{align*}
Using this inequality, and the conditional Jensen inequality, we obtain
\begin{align*}
& \Erw{ \left( \sum_{i=1}^N \abs{yC_iR_i}^2 \right)^q - \sum_{i=1}^N \abs{yC_i
R_i}^{2q} } \\
\le & \Erw{  \sum_{(j_1, \dots, j_N)\in A_p(N)} {p \choose j_1, \dots, j_N}
\abs{yC_1R_1}^{2\gamma j_1} \cdots \abs{yC_N R_N}^{2\gamma j_N}  } \\
\le & \Erw{  \Erw{ \left. \sum_{(j_1, \dots, j_N)\in A_p(N)} { p \choose j_1, \dots, j_N}
\abs{yC_1R_1}^{2 \gamma j_1} \cdots \abs{yC_N R_N}^{2 \gamma j_N} \right| C_1,
\dots. C_N } } \\
 \le & \Erw{  \sum_{(j_1, \dots, j_N)\in A_p(N)} {p \choose j_1, \dots, j_N} \norm{C_1}^{2 \gamma
 j_1} \cdots \norm{C_N} ^{2 \gamma j_N} \E \abs{R}^{2 \gamma j_1} \cdots \E
 \abs{R}^{2 \gamma j_N}  } \\
 \le & \Erw{ \sum_{(j_1, \dots, j_N)\in A_p(N)} {p \choose j_1, \dots, j_N} \norm{C_1}^{2 \gamma j_1}
 \cdots \norm{C_N} ^{2 \gamma j_N} \left( \E \abs{R}^{2\gamma(p-1)}\right)^p   }
 \\ = & \left( \E \abs{R}^{2\gamma(p-1)}\right)^p \Erw{  \left[
 \sum_{i=1}^N \norm{C_i}^{2\gamma} \right]^p - \sum_{i=1}^N
 \norm{C_i}^{2\gamma p}  } \\ \le & \left( \E \abs{R}^{2\gamma(p-1)}\right)^p
 \Erw{ \left( \sum_{i=1}^N \norm{C_i}^{2\gamma} \right)^p } \\
 \le & \left( \E \abs{R}^{2\gamma(p-1)}\right)^p N^p \Erw{\norm{C}^{2q}}\\
 \le & \left( \E \abs{R}^{\frac{s/2}{\lceil s/2 \rceil}2(\lceil s/2 \rceil
-1)}\right)^{s/2} N^{s/2} \E \norm{C}^s .
\end{align*}

\medskip

\textsc{Case} $s \le 2$:
For $H_1$, by the usual inequality $\abs{a^{s/2}-b^{s/2}} \le \abs{a-b}^{s/2}$
 and the  Cauchy-Schwartz inequality:
$$ H_1 \le \E \sum_{k \neq l} \abs{yC_k R_k}^\frac{s}2 \abs{yC_l R_l}^\frac{s}2
\le N^2 \left( \E \norm{C}^s \right) \left( \E \abs{R}^\frac{s}2
\right)^2 .$$

Finally
\begin{align*}
H_2 \le & \Erw{ \sum_{i=1}^N \abs{yC_i R_i}^s - \left( \max_{1 \le i \le N}
\abs{yC_i R_i} \right)^s} + \Erw{ \left( \sum_{i=1}^N \abs{yC_i R_i}^2
\right)^\frac{s}{2} - \left( \max_{1 \le i \le N} \abs{yC_i R_i}^2
\right)^\frac{s}2} \\
\le & 2 \Erw{ \sum_{i=1}^N \abs{yC_i R_i}^s - \left( \max_{1 \le i \le N}
\abs{yC_i R_i} \right)^s} \le 2\eqref{E2}.
\end{align*}

\end{proof}



\bigskip

\noindent\textsc{D. Buraczewski, E. Damek, M. Mirek.\\
Uniwersytet Wroclawski, Instytut Matematyczny, Pl. Grunwaldzki 2/4, 50-384 Wroclaw}\\
e-mail: \verb+ dbura@math.uni.wroc.pl+, \verb+ edamek@math.uni.wroc.pl+, \verb+ mirek@math.uni.wroc.pl+ \\

\noindent\textsc{S. Mentemeier. \\
University of M\"unster, Institut f\"ur Mathematische Statistik, Einsteinstra\ss e 62, 48149 M\"unster.}\\
e-mail: \verb+ mentemeier@uni-muenster.de+

\end{document}